\documentclass{article}
\usepackage{graphicx} 
\usepackage[utf8]{inputenc}
\usepackage[top=1.5in,bottom=1.5in,right=1.5in,left=1.5in]{geometry} 

\usepackage{authblk} 
\usepackage{cite}

\usepackage{algorithm}
\usepackage[noend]{algpseudocode}
\usepackage{dsfont}

\usepackage{amsmath,amssymb,amsthm,xcolor}
\usepackage{hyperref, cleveref}
\usepackage{tikz-cd}

\usepackage{nicematrix}

\newtheorem{theorem}{Theorem}[section] 
\newtheorem{corollary}[theorem]{Corollary} 
\newtheorem{proposition}[theorem]{Proposition} 
\newtheorem{lemma}[theorem]{Lemma}

\theoremstyle{definition}
\newtheorem{definition}[theorem]{Definition}
\newtheorem{example}[theorem]{Example}
\newtheorem{remark}[theorem]{Remark}

\DeclareMathOperator{\val}{\nu}
\DeclareMathOperator{\initial}{in}

\DeclareMathOperator{\lc}{lc}

\DeclareMathOperator{\rk}{rk}
\DeclareMathOperator{\Trop}{Trop}
\DeclareMathOperator{\Tropw}{Trop_\mathcal{C}^+}

\DeclareMathOperator{\Log}{Log}

\DeclareMathOperator{\rowspan}{rowspan}
\DeclareMathOperator{\Id}{Id}
\DeclareMathOperator{\diag}{diag}
\DeclareMathOperator{\Cone}{Cone}
\DeclareMathOperator{\Conv}{Conv}

\DeclareMathOperator{\relint}{relint}
\DeclareMathOperator{\mult}{mult}

\newcommand{\GG}{\mathbb{G}}
\newcommand{\Z}{\mathbb{Z}}
\newcommand{\R}{\mathbb{R}}
\newcommand{\C}{\mathbb{C}}
\newcommand{\RR}{\mathcal{R}}
\newcommand{\CC}{\mathcal{C}}

\newcommand{\LL}{\mathcal{L}}

\newcommand{\inw}{\operatorname{in}_w}

\title{Computing positive tropical varieties and lower bounds on the number of positive roots}

\author[1]{Kemal Rose} 
\author[2]{Máté L. Telek \thanks{corresponding author}}

\affil[1]{KTH Royal Institute of Technology, Stockholm, Sweden. kemalr@kth.se}
\affil[2]{Max Planck Institute for Mathematics in the Sciences, Leipzig, Germany. mate.telek@mis.mpg.de}

\date{} 

\begin{document}
\maketitle

\begin{abstract}
We present two effective tools for computing the positive tropicalization of an algebraic variety. First, we outline conditions under which the initial ideal can be used to compute the positive tropicalization, offering a real analogue to the Fundamental Theorem of Tropical Geometry. Additionally, under certain technical assumptions, we provide a real version of the Transverse Intersection Theorem. Building on these results, we propose an algorithm to compute a combinatorial bound on the number of positive real roots of a system of parametrized polynomial equations. Furthermore, we discuss how this combinatorial bound can be applied to study the number of positive steady states of chemical reaction networks.
    \vskip 0.1in
	
\end{abstract}

\section{Introduction}

Systems of polynomial equations are fundamental in various scientific disciplines, serving as powerful tools to model and solve real-world problems. Methods from applied algebraic geometry \cite{cox2005using} and nonlinear algebra \cite{MS_invitation,NonLinAlgPaper} are used to study the \emph{algebraic varieties} defined by these polynomials. 
Tropical geometry provides additional tools, which simplify the study of solution sets by transforming curvy algebraic varieties into piece-wise linear polyhedral complexes, known as \emph{tropical varieties}. 
Nowadays, tropical geometry is a well-established field commonly used to study solutions of polynomial equations over algebraically closed fields~\cite{maclagan2015introduction,TropAlgGeo_book}. For example, methods from tropical geometry are employed in general-purpose solvers in numerical algebraic geometry~\cite{HuberSturmfels,LEYKIN2019173, HCBreiding,HelminckHenrikssonRen}.

Many scientific applications lead to the study of positive solutions of polynomial equations. This includes scattering amplitudes in particle physics \cite{lam2022invitationpositivegeometries,ArkeniHamedLam}, statistical models in phylogenetics~\cite{PhyloGen}, determining steady states in chemical reaction networks~\cite{DickensteinInvitation,FeliuRoleofAlg}, and identifying Nash equilibria in game theory~\cite{GameTheory}. These scientific problems motivate the development of tools for studying positive solutions. Recently, there has been increasing attention for the tropicalization of the positive part of algebraic varieties \cite{SpeyerWilliams::PosGrassmannian,ARDILA2006577,Tabera::Bases,Brandenburg::TropicalPositivity,BENDLE2024102224,akhmedova2023tropicalamplituhedron}, and more generally, of semi-algebraic sets \cite{Alessandrini::LogLimit,Allamigeon::TropSpecta,JellScheidererYu::RealTrop,Blekherman::Moments}.

The computation of positive tropicalization is challenging and only feasible in specific cases. We identify two main technical obstacles. On the one hand, there is a lack of tools to certify containment of a vector in the positive tropicalization. On the other hand, tropicalization in general does not commute with forming intersections. Specifically, the tropicalization of the intersection of two varieties might be strictly contained in the
intersection of their tropicalizations.
Over algebraically closed base-fields, the first obstruction is addressed by the Fundamental Theorem of Tropical Geometry \cite[Theorem 3.2.3]{maclagan2015introduction}, which characterizes containment in the tropicalization in terms of initial ideals.
The second obstruction is addressed by the Transverse Intersection Theorem \cite{BOGART200754,LiftingInters,maclagan2015introduction},
which provides a genericity condition under which the tropicalization of the intersection of two varieties is the intersection of their tropicalizations. Under a restrictive technical assumption, in this paper we prove real analogues of both statements.
\Cref{thm: multiplicity one cones are positive} is an analogue of 
the Fundamental Theorem of Tropical Geometry
 and \Cref{thm: weakly positive transverse intersection} is an analogue of the Transverse Intersection Theorem.

Building on these results, we investigate polynomial equation systems with coefficients in a real closed field $\mathcal{R}$ with a non-Archimedean valuation.
In \Cref{thm: tropical equation solving}, we tropicalize the set of positive solutions and give an explicit combinatorial description under appropriate genericity conditions.
As a special case, for a system of $n$ equations in $n$ variables $x_1, \dots ,x_n$
\begin{align}
\label{Eq:FREE}
 f_i(x) = \sum_{j=1}^r c_{i,j} x^{\alpha_j} = 0, \qquad i = 1, \dots, n
\end{align}
we obtain a combinatorial lower bound on the number of positive solutions.
These lower bounds are expressed as the intersection $\Trop^+(\ker(C)) \cap \rowspan(A)$ of the positive tropicalization of a linear space $\Trop^+(\ker(C))$ with the rowspan of $A$.
Here,  $C = (c_{i,j}) \in \mathcal{R}^{n \times r}$ denotes the coefficient matrix of \eqref{Eq:FREE} and the matrix $A \in \mathbb{Z}^{n \times r}$ contains the exponents $\alpha_1, \dots , \alpha_r$ as columns.

In Section~\ref{Sec::LinearSpaces}, we present an algorithm for computing the positive tropicalization of linear spaces, such as $\Trop^+(\ker(C))$. We also discuss how this algorithm can be implemented in \texttt{Oscar.jl} \cite{OSCARbook}, thereby making the combinatorial bound on the number of positive solutions to~\eqref{Eq:FREE} computable.

As a direct consequence, for sufficiently small values of $t> 0$, we provide an algorithm for computing lower bounds on the number of positive real roots of a polynomial equation system of the form
\begin{align}
\label{Eq:System}
 N \diag(t^h) x^A = 0,
\end{align}
 where $N  \in \mathbb{R}^{n\times r}$, $h \in \mathbb{R}^r$, $t^h = (t^{h_1},\dots,t^{h_r})$, $A = (\alpha_1 | \dots| \alpha_r)  \in \mathbb{Z}^{n \times r}$, and $x^A = (x^{\alpha_1}, \dots x^{\alpha_r})^\top$. Following \cite{RenHelminck,FeliuDim}, we call the equation system \eqref{Eq:System} \emph{vertically parametrized}.
This name reflects that the coefficient matrix $N \diag(t^h)$ is obtained from $N$ by multiplying its $j$-th column by the same parameter $t^{h_j}$. In Theorem~\ref{Prop::LowerBoundVertic}, we show that for generic $h \in \mathbb{R}^r$ and sufficiently small $t > 0$, the number of positive real solutions of~\eqref{Eq:System}  is at least the cardinality of
\begin{align}
\label{Eq:PosTropInt}
 (\Trop^+(\ker(N))-h) \cap \rowspan(A).
\end{align}

We view Theorem~\ref{Prop::LowerBoundVertic} as an extension of Viro's method \cite{Viro_Dissertation, Bernd::Patchworking} to vertically parameterized systems. In Remark~\ref{Remark:ViroforVertParam}, we discuss why Viro's method cannot be directly applied to vertically parameterized systems and how Theorem~\ref{Prop::LowerBoundVertic} overcomes this obstacle. Theorem~\ref{Prop::LowerBoundVertic} can also be seen as a real version of \cite[Proposition 6.5]{RenHelminck}, where the authors showed that the generic root count of~\eqref{Eq:System} in  $(\mathbb{C}^*)^n$ is given by a tropical intersection number, which generalizes the well-known BKK-bound~\cite{Bernstein,Kushnirenko}. It should also be noted that in \cite{PolyhedralMethodSparsePositive}, a bound on the number of positive real solutions of~\eqref{Eq:System}, for small values of $t > 0$, was provided in terms of \emph{positively decorated simplices} of a regular triangulation of the columns of~$A$. In \Cref{Prop:InjMapPosDec}, we prove that the lower bound given by the number of positively decorated simplices is always at most the number of points in \eqref{Eq:PosTropInt}. Moreover, in Example~\ref{Eq:StircExample}  we present a vertically parametrized system where strict inequality occurs.

One of the main motivations for studying equation systems of the form~\eqref{Eq:System} comes from chemical reaction network theory \cite{DickensteinInvitation, FeliuRoleofAlg}, where the steady states of the network are given by vertically parametrized systems augmented with some linear equations, called conservation laws. In Section~\ref{Sec:CRNT}, we showcase how Theorem~\ref{Prop::LowerBoundVertic} can be applied in this setting. A comprehensive study of applying Theorem~\ref{Prop::LowerBoundVertic} to biologically relevant reaction networks will be part of future work.

The paper is organized as follows. In Section~\ref{Sec::RealTrop}, we review notions and techniques from tropical geometry and discuss fields with non-Archimedean valuation. In Section~\ref{Section:PositiveTrop}, we focus on positive tropicalization and provide a real analogue of the Fundamental Theorem of Tropical Geometry (\Cref{thm: multiplicity one cones are positive}). In Section~\ref{Sec::Comp}, we discuss how to compute the positive tropicalization of algebraic varieties, in particular linear spaces such as $\Trop^+(\ker(N))$, and prove a real version of the Transverse Intersection Theorem (Theorem~\ref{thm: weakly positive transverse intersection}). In Section~\ref{Sec:PosSolsTrop}, we apply tropical methods to study the positive solutions of a polynomial equation system as in~\eqref{Eq:FREE} with coefficients in a real closed field. We conclude the paper in Section~\ref{Section::SparsePolys}, where we provide lower bounds for vertically parametrized systems, compare this result with the bound based on positively decorated simplices from~\cite{PolyhedralMethodSparsePositive}, and using our tropical method we compute lower bounds on the number of positive steady states of reaction networks.

\paragraph{\textbf{Notation. }}
For a field $K$, we write $K^* = K \setminus \{0\}$. We denote by $V(I)$ the very affine vanishing locus of an ideal $I \subseteq K[x_1^\pm, \dots, x_n^\pm]$. If the ideal $I$ is generated by polynomials $f_1, \dots ,f_k$, we write $V(f_1,\dots,f_k) = V(I)$. For variables $x_1, \dots , x_n$ and $\alpha \in \mathbb{Z}^n$, we use the shorthand notation $x^\alpha = x_1^{\alpha_1}\dots x_n^{\alpha_n}$. For two vectors $v,w \in \mathbb{R}^n$, $v\cdot w$ denotes the Euclidean scalar product, and $v \ast w$ the  coordinate-wise Hadamard product. The transpose of a matrix $M$ will be denoted by $M^\top$.

\section{Background on tropical geometry}
\label{Sec::RealTrop}
\subsection{Non-Archimedean valued fields}

To define the tropicalization of an algebraic variety, we work over a field equipped with a non-Archimedean valuation. To set the stage, we begin by discussing such fields and two slightly different notions of positivity.
\color{black}
Throughout, $\RR$ denotes a real closed field with a non-trivial, non-Archimedean valuation $\nu \colon \mathcal{R}^* \to \mathbb{R}$ that is compatible with the order of $\RR$. That is,
\[\forall a, b \in \RR \ : \ 0 \leq  a \leq b \implies \nu(a) \geq \nu(b).\] 
We denote by $\RR_{>0}$ the set of positive elements of $\RR$ and by 
\[ \RR^\circ := \big\{ x \in \RR \mid  \nu(x) \geq 0 \big\} \quad \text{and} \quad \RR^{\circ \circ} := \big\{ x \in \RR \mid  \nu(x) > 0 \big\} \]
the valuation ring and its unique maximal ideal. The residue field $\RR^\circ/ \RR^{\circ \circ}$ is denoted by~$\widetilde{\mathcal{R}}$. Note that the ordering of $\RR$ induces an ordering of $\widetilde{\mathcal{R}}$.
We write $\overline{x}$ for the image of $x \in \RR^\circ$ under the canonical map $\RR^\circ \to \widetilde{\mathcal{R}}$.

The valuation of $\RR$ uniquely extends to its algebraic closure $\CC$, whose residue field  will be denoted as  $\widetilde{\CC}$.
Since $\mathcal{C}$ is algebraically closed, there exists a section of the valuation map \cite[Lemma 2.1.15]{maclagan2015introduction}, that is, there exists a map $\psi\colon (\nu(\mathcal{C}^*),+) \to (\mathcal{C}^*,\cdot)$ such that $\val(\psi(w)) = w$ for all $w \in \nu(\mathcal{C}^*)$. We denote the element $\psi(w) \in \mathcal{C}^*$ by $t^w$. The section $\psi$ can be choosen such that its restriction to $\nu(\mathcal{R}^*)$ induces a section  $(  \val(\mathcal{R}^* ), +) \to (\mathcal{R}_{>0}, \cdot ), \ w \mapsto t^w$.

The reader might think of $\mathcal{R}$ as the field of \emph{formal real Puiseux series} ~$\mathbb{R}\{\!\{t \}\!\}$ and its algebraic closure $\mathcal{C} = \mathbb{C}\{\!\{t \}\!\}$, the field of \emph{formal complex Puiseux series}. These fields have commonly been used in the literature to study positive tropicalization \cite{SpeyerWilliams::PosGrassmannian,Tabera::Bases,Brandenburg::TropicalPositivity}. 

To be concrete, we recall that a non-zero \emph{formal Puiseux series} is a series of the form
\[ x(t)=\sum_{k=k_0}^{\infty} c_k t^{k / N} \text { for some } k_0 \in \mathbb{Z},\, N \in \mathbb{N},\, c_k \in \mathbb{C}, \, c_{k_0} \neq 0. \]
Its valuation is defined as the smallest exponent of $t$ in $x(t)$ with non-zero coefficient, that is, $\val(x(t)) = \tfrac{k_0}{N}$. The coefficient of the term with smallest exponent is called the \emph{leading coefficient} of $x(t)$, denoted by $\lc (x(t)) = c_{k_0}$.

A formal Puiseux series is \emph{real} if all the coefficients $c_k$ are real numbers. It is known that the field of formal real Puiseux series is a real closed field \cite[Theorem 2.91]{basu2007algorithms}. An element $x(t) \in \mathbb{R}\{\!\{t \}\!\}^*$ is \emph{positive} if $\lc(x(t))$ is positive. In \cite{SpeyerWilliams::PosGrassmannian}, the authors considered a slightly different notion of positivity for the elements of $\mathbb{C}\{\!\{t \}\!\}$. In their definition, an element $x(t) \in \mathbb{C}\{\!\{t \}\!\}$ is positive if $\lc(x(t))$ is real and positive. In the current paper, we call such elements \emph{weakly positive}. Note that for weakly positive $x(t)$, only the leading coefficients must be real; the other coefficients can be complex numbers.

We generalize the notion of weak positivity to any real closed field $\RR$ with a compatible non-trivial, non-Archimedean valuation. We say that an element $x \in \CC$ in the algebraic closure is \emph{weakly positive} if the reduction of $y = t^{-\nu(x)}x \in \CC^\circ$ is real and positive, that is, $\overline{y} \in \widetilde{\RR}_{>0}$.
This definition does not depend on the choice of the section $t^{-\nu(x)} \in \RR_{>0}$, since a different choice replaces the reduction $\overline{y}$ with a positive multiple.
We denote by $\CC_{>0}$ the set of weakly positive elements of $\CC$. By definition, we have
\begin{align}
\label{Eq:DiffPosIncl}
\RR_{>0} \subsetneq \CC_{>0}.
\end{align}

Throughout the paper 
we will work over different base-fields. 
 In Section~\ref{Sec::LinearSpaces}, we will use $\mathbb{R}\{\!\{t \}\!\}$  to compute the positive tropicalization of linear spaces defined over $\mathbb{R}$. 
 In Section~\ref{Section::SparsePolys}, we will work with the \emph{Hardy field} $H(\mathbb{R}_{\text{an}^*})$, as it allows translating statements defined over $H(\mathbb{R}_{\text{an}^*})$ to statements over $\mathbb{R}$. For example, a polynomial equation system has the same number of positive solutions in $H(\mathbb{R}_{\text{an}^*})$  as the number of positive solutions of the induced system in $\mathbb{R}$ when evaluating the coefficients of the polynomials for small values of~$t$ (cf. Theorem~\ref{Thm::NumOfSols}). In the context of tropicalization of semi-algebraic sets, the Hardy field appeared in \cite{Alessandrini::LogLimit}.

The \emph{Hardy field} $H(\mathbb{R}_{\text{an}^*})$ \cite[Section 4.1]{Alessandrini::LogLimit} is a modified version of the field of formal real Puiseux series, and is a well-studied object in model theory and o-minimal geometry. 
In Appendix \ref{AppendixA}, we provide a brief introduction to non-experts in o-minimal geometry, and for further details, we refer to \cite{Alessandrini::LogLimit,Ebbinghaus::Logic}. Here, we proceed by describing the elements of $H(\mathbb{R}_{\text{an}^*})$. A \emph{generalized real Puiseux series} is a power series of the form
\begin{align}
\label{Eq::GenRealPuisSeries}
 x(t) = t^r \sum_{\alpha \in \mathbb{R}_{\geq 0}} c_\alpha t^\alpha, \text{ where } r \in \mathbb{R}, \, c_\alpha \in \mathbb{R}, \, c_0 \neq 0.
\end{align}
 We say that $x(t)$ is (absolutely) \emph{locally convergent} if there exists $\epsilon \in \mathbb{R}_{>0}$ such that 
\[ \sum_{\alpha \in \mathbb{R}_{\geq 0}} \lvert c_\alpha \rvert \epsilon^\alpha < \infty.\]
Here $ \lvert c_\alpha \rvert$ denotes the usual Archimedean absolute value of $c_\alpha \in \mathbb{R}$.
The valuation of $x(t)$ is $\nu(x(t)) := r$, and $x(t)$ is defined to be positive, $x(t) > 0$, if its leading coefficient $c_0$ is positive. Two locally convergent generalized real Puiseux series $x(t),\, y(t)$ are equivalent if and only if there exists $\epsilon \in \mathbb{R}_{>0}$ such that for all $\tau \in (0,\epsilon)$ both $x(\tau)$ and $y(\tau)$ converge, and $x(\tau) = y(\tau)$. The Hardy field, denoted by $H( \mathbb{R}_{\text{an}^*})$, is the set of the equivalence classes $[x(t)]$ of such locally convergent generalized real Puiseux series under this relation.

The Hardy field possesses several favorable properties that make it a convenient choice for real tropical geometry.
It is known that $H( \mathbb{R}_{\text{an}^*})$ is a real closed field of rank one extending $\mathbb{R}$  \cite[Theorem 5.8]{Coste::oMinimal}. An element $[x(t)] \in H( \mathbb{R}_{\text{an}^*})$ is positive if and only if $x(t) > 0$. Moreover, $H( \mathbb{R}_{\text{an}^*})$ has a non-trivial valuation given by $\nu([x(t)]) = \nu(x(t))$ \cite[Section 4.1]{Alessandrini::LogLimit}. In particular, $\nu( H( \mathbb{R}_{\text{an}^*})^*) = \mathbb{R}$. 

Each polynomial $f \in H( \mathbb{R}_{\text{an}^*})[x_1^{\pm}, \dots , x_n^{\pm}]$ induces a polynomial with real coefficients $f_\tau \in \mathbb{R}[x_1^{\pm}, \dots , x_n^{\pm}]$ for every sufficiently small $\tau \in \mathbb{R}_{>0}$. The polynomial $f_\tau$ is obtained from $f$ by evaluating its coefficients at $t = \tau$.

\begin{theorem}
    \label{Thm::NumOfSols} Let $f_1,\dots, f_n \in H(\mathbb{R}_{\text{an}^*})[x_1^{\pm}, \dots , x_n^{\pm}]$. The equation system $f_1 = \dots = f_n = 0$ has $k$ solutions in $H(\mathbb{R}_{\text{an}^*})^n_{>0}$ if and only if there exists $\epsilon \in \mathbb{R}_{>0}$ such that for all $\tau \in (0,\epsilon)$ the equation system $f_{1,\tau} = \dots = f_{n,\tau} = 0$ has $k$ solutions in $\mathbb{R}^n_{>0}$. 
\end{theorem}

\begin{proof}
    Since both statements can be expressed as a definable family of formulas, the result follows from \cite[Proposition 5.9]{Coste::oMinimal}.
\end{proof}

\begin{remark}
    Instead of the Hardy field $H(\mathbb{R}_{\text{an}^*})$, one could also use the field of real algebraic Puiseux series $\mathbb{R}\{t\}$, which is the real closure of the rational function field $\mathbb{R}(t)$ \cite[Corollary 2.98]{basu2007algorithms}.
    The analogue of \Cref{Thm::NumOfSols} also holds true for $\mathbb{R}\{t\}$ \cite[Propositon 3.17]{basu2007algorithms}.
    However, $\nu(\mathbb{R}\{t\}^*) \subsetneq \mathbb{R} $, while $\nu(H(\mathbb{R}_{\text{an}^*})^*) = \mathbb{R}$, and for this reason we use the Hardy field  $H(\mathbb{R}_{\text{an}^*})$ in this work.
\end{remark}

\subsection{Tropicalization of algebraic varieties}
\label{SubSec:TropAlg}
To establish a connection between the positive tropicalization and the tropicalization of algebraic varieties, we first recall the definition and some important properties of the latter.

With a slight abuse of notation, we write $\val \colon (\mathcal{C}^{*})^n\to \mathbb{R}^n$ for the coordinate-wise valuation map.
Throughout the article, all varieties $V = V(I) \subseteq (\mathcal{C}^{*})^n$ will be very affine varieties with defining ideal $I \subseteq \mathcal{C}[x_1^\pm, \dots, x_n^\pm]$ in the Laurent polynomial ring.
We define the \emph{tropicalization $\Trop(V)$} as the closure of $\val(V)$ in the Euclidean topology of $\mathbb{R}^n$ \cite[Theorem 3.2.3]{maclagan2015introduction}. One of the earliest instances of tropical geometry traces back to 1971, when Bergman investigated the logarithmic limit of complex algebraic varieties \cite{loglims}. For $t > 0$, denote
 \[\Log _t:\left(\mathbb{C}^*\right)^n \to \mathbb{R}^n, \quad x \mapsto\left(\log _t\left(\left|x_1\right|\right), \ldots, \log _t\left(\left|x_n\right|\right)\right), \]
 the coordinatewise logarithm map. The \emph{logarithmic limit} of a set $X \subseteq (\C^*)^n$ is traditionally  defined as the limit of the family of sets $\Log_t(X)$ as $t \to \infty$, where the limit should be understood in the sense of \cite[Section 2.1]{Alessandrini::LogLimit}.
 In this work, we use the min-convention for defining tropical varieties, therefore we consider the following definition of the logarithmic limit of $X$
 \begin{align}
     \label{Eq::LogLimitBergman}
     \LL(X) := - \lim_{t \to \infty} \Log_t(X).
 \end{align}
In other words, the logarithmic limit $\LL(X)$ is the reflection of the logarithmic limit set from~\cite{loglims} around the origin. The connection to tropical geometry is due to the following result.

\begin{theorem} 
\label{Thm::ComplexLogLimit}
(\cite{Jonsson:Amoeba}, see also \cite[Theorem 1.4.2]{maclagan2015introduction})
     Let $V_\mathbb{C} \subseteq (\mathbb{C}^*)^n$ be a very affine algebraic variety defined by an ideal $I \subseteq \C[x_1^\pm, \dots ,x_n^\pm]$. Let $\mathcal{C}$ be an algebraically closed extension of $\C$ with non-trivial valuation. Then
     \[ \Trop(V) = \LL(V_\C),\]
     where $V \subseteq (\mathcal{C}^*)^n$ denotes the very affine variety defined by $I$. 
\end{theorem}

Throughout the paper, we will write $\Trop(V_{\mathbb{C}})$ for the logarithmic limit, as defined in ~\eqref{Eq::LogLimitBergman}, of a very affine variety $V_\mathbb{C} \subseteq (\mathbb{C}^*)^n$. By Theorem~\ref{Thm::ComplexLogLimit}, $\Trop(V_{\mathbb{C}})$ is equal to the tropicalization of $V$, where $V$ is the Zariski closure of $V_{\mathbb{C}}$ in $(\mathcal{C}^*)^n$ for any algebraically closed extension $\mathcal{C}$ of $\mathbb{C}$ with a non-trivial valuation.  

Next, we recall a characterization of $\Trop(V)$ in terms of \emph{initial ideals}. For $w \in \mathbb{R}^n$ and $f = \sum c_\alpha x^\alpha \in \mathcal{C}[x_1^{ \pm 1}, \ldots, x_n^{ \pm 1}]$, we define the initial form 
\[\inw(f)=\sum_{\alpha \colon \nu\left(c_{\alpha}\right)+ w \cdot \alpha =W} \overline{t^{-\nu\left(c_{\alpha}\right)} c_{\alpha}} \phantom{\cdot} x^{\alpha},  \]
where $W=\min \left\{\nu\left(c_\alpha\right)+w \cdot \alpha \right\}$. The \emph{initial ideal}, denoted as $\inw(I)$, of an ideal $I \subseteq \mathcal{C}[x_1^\pm, \dots , x_n^\pm]$ is defined as the ideal generated by the initial forms $\inw(f), \, f \in I$ in the Laurent polynomial ring $\widetilde{\CC}[x_1^\pm, \dots , x_n^\pm]$ over the residue field.
Furthermore, we define the initial degeneration $\inw(V(I))$ to be the scheme $\operatorname{spec}\left(\widetilde{\CC}[x_1^\pm, \dots , x_n^\pm] / \inw(I)\right)$ defined by $\inw(I)$.
By the Fundamental Theorem of Tropical Algebraic Geometry, one can verify membership of a point in $\Trop(V(I))$ using the initial ideal and initial degeneration.

\begin{theorem}\cite[Fundamental Theorem of Tropical Algebraic Geometry, Theorem 3.2.3]{maclagan2015introduction}
\label{Thm:FundamentalThm}
    Let $I \subseteq \mathcal{C}[x_1^{\pm}, \dots , x_n^{\pm}]$ be an ideal. For $w \in \mathbb{R}^n$, the following are equivalent
\begin{itemize}
    \item[(i)] $w \in \Trop(V(I))$,
    \item[(ii)] $\inw(V(I)) \neq \emptyset$, 
    \item[(iii)] $\inw(I)$ contains no monomials.
\end{itemize}
\end{theorem}

\begin{example}
\label{Ex::ComplexTrop}
    Consider the polynomial $f = x_1^2-2x_1+1 + x_2 \in \mathbb{C}\{\!\{t \}\!\}[x_1,x_2]$. We depict its Newton polytope, that is, the convex hull of the exponent vectors $(2,0),(1,0),(0,0),(0,1)$ in Figure~\ref{FIG1}(a).
    
    As the coefficients of $f$ are in $\mathbb{C}$, the initial form coincides with the truncation $f_{|F}$ to the face $F$ with inner normal vector $w$. For example, for $w = (0,1)$ we have
    \begin{align}
    \label{Ex:InitForm}
     \inw(f) =  x_1^2 - 2 x_1 + 1 = f_{|F}(x_1,x_2)
     \end{align}
    where $F = \Conv((0,0),(2,0))$. By the truncation $f_{|F}$, we simply mean the sum of all monomials whose exponent vectors are contained in the face $F$.

    By computing all initial forms and using Theorem~\ref{Thm:FundamentalThm}, we find that the tropicalization of $V(f)$ consists of the rays spanned by the vectors $(0,1),\,(0,1),\,(-1,-2)$, see Figure~\ref{FIG1}(b). These rays form a polyhedral fan that coincides with the $(n-1)$-skeleton of the inner normal fan of the Newton polytope of $f$. This is a general phenomenon that occurs for the tropicalization of hypersurfaces defined by polynomials whose coefficients have zero valuation \cite[Theorem 3.1.3]{maclagan2015introduction}.

    By Theorem~\ref{Thm::ComplexLogLimit}, the tropicalization of $V(f)$ equals the limit of the sets $-\Log_t(V(f))$ as $t \to \infty$. Figure~\ref{FIG1}(c) shows   $-\Log_t(V(f))$ for $t = e$ (Euler's number). Intuitively, one can think of the logarithmic limit as shrinking the sets $-\Log_t(V(f))$.
\end{example}

Next, we recall the definition of multiplicities of cells in $\Trop(V(I))$  \cite[Section 3.4]{maclagan2015introduction},\cite[Definition 2.6]{Bossinger2017}.
By \cite[Theorem 2.2.1]{Speyer_Thesis}, there exists a polyhedral complex $\Sigma$ with support $ \Trop(V(I))$ such that $\inw(I)$ is constant for each $w \in \relint(\sigma)$ for all $\sigma \in \Sigma$.
We associate a \emph{multiplicity} with each maximal cell $\sigma \in \Sigma$ as follows. Let $\text{Ass}^\text{min}(\inw(I))$ be the set of minimal associated primes of $\inw(I)$ that do not contain any monomial, and denote by $\mult(P,\inw(I))$ the multiplicity of a minimal prime $P \in \text{Ass}^\text{min}(\inw(I))$ (cf. \cite[Definition 3.4.1]{maclagan2015introduction}, \cite[Chapter 3]{eisenbud1995commutative}). The multiplicity of a maximal cell $\sigma \in \Sigma$ is defined as
\begin{align}
\label{Eq:Multiplicity}
 \mult(\sigma) := \sum_{P \in\text{Ass}^\text{min}(\inw(I))} \mult(P,\inw(I)), \qquad \text{for any } w \in \relint(\sigma).  
 \end{align}

\begin{figure}[t]
\centering
\begin{minipage}[h]{0.3\textwidth}
\centering
\includegraphics[scale=0.4]{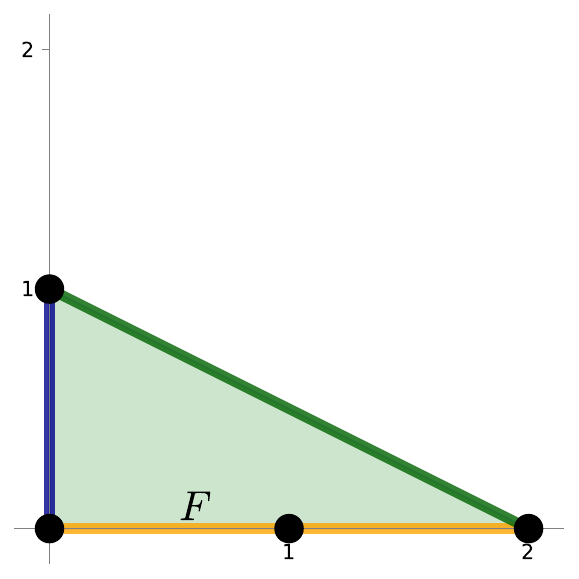}

{\small (a) Newton polytope of $f$}
\end{minipage}
\hspace{10 pt}
\begin{minipage}[h]{0.3\textwidth}
\centering
\includegraphics[scale=0.4]{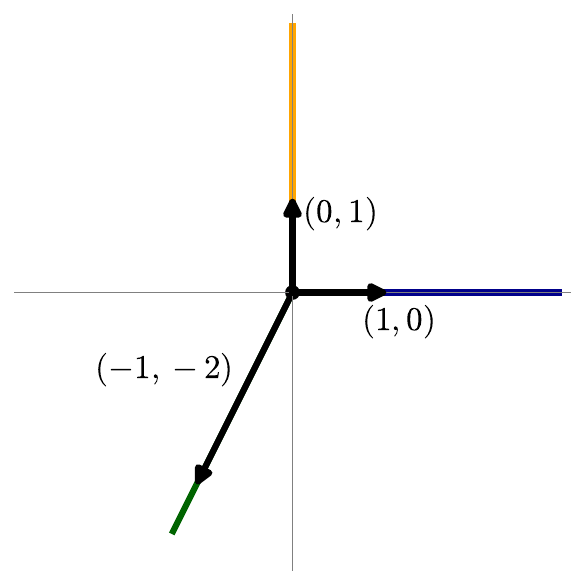}

{\small (b) $\Trop(V(f))$}
\end{minipage}
\hspace{10 pt}
\begin{minipage}[h]{0.3\textwidth}
\centering
\includegraphics[scale=0.4]{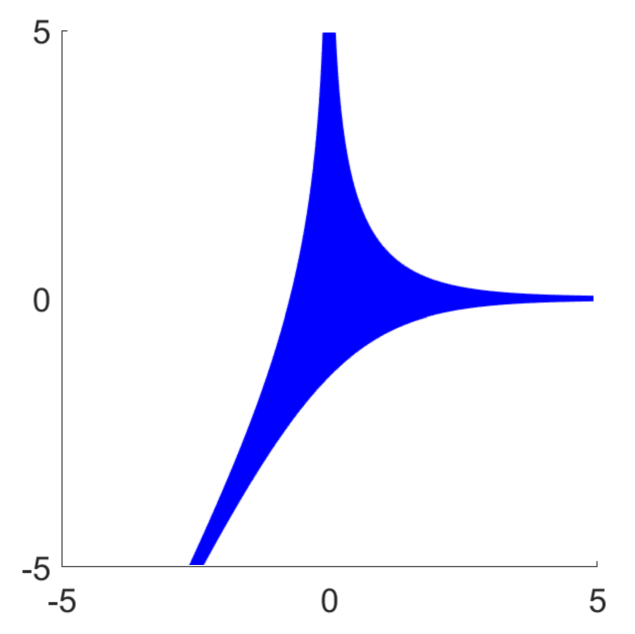}

{\small (c) $-\Log_e(V(f))$}
\end{minipage}

\caption{{\small  Illustration of Example~\ref{Ex::ComplexTrop} with $f = x_1^2 - 2x_1 +1 - x_2$}}\label{FIG1}
\end{figure}

Following \cite{LiftingInters}, we call a point $w \in \mathbb{R}^n$ a \emph{simple point} of the complex tropicalization $\Trop(V(I))$ if $w$ lies in the relative interior of a maximal cell of multiplicity one.

\begin{remark}
\label{Rmk:NotSimplePoint}
    If $I \subseteq \mathcal{C}[x_1^{\pm}, \dots , x_n^{\pm}]$ is a principal ideal generated by $f \in \mathcal{C}[x_1^{\pm}, \dots , x_n^{\pm}]$, then by \cite[Lemma 3.4.6]{maclagan2015introduction} the multiplicity of a maximal cell $\sigma$ is the lattice length of the edge in the polyhedral subdivision of the Newton polytope of $f$ that is dual to $\sigma$. In particular, if $\sigma$ has multiplicity one, the initial form $\inw(f)$ is a binomial for any $w \in \relint(\sigma)$.
    For example, consider the polynomial $f$ from Example~\ref{Ex::ComplexTrop}. The cell $\sigma = \Cone((1,0))$ is dual to the edge $\Conv((0,0),(0,1))$, which has lattice length one. Thus, $\mult(\sigma) = 1$, $w = (1,0)$ is a simple point of $\Trop(V(I))$. Moreover, the initial form $\text{in}_{(1,0)}(f) = 1 + x_2$ is a binomial. 

    The edge $F = \Conv((0,0),(2,0))$ dual to $\Cone((0,1))$ has lattice length $2$. Therefore, $(0,1)$ is not a simple point of $\Trop(V(I))$. Note that the initial form \eqref{Ex:InitForm} is not a binomial.
\end{remark}

The property that initial ideals corresponding to maximal cells of multiplicity one are binomial extends beyond the context discussed in Remark~\ref{Rmk:NotSimplePoint}. Next we state this result, which will play a crucial role in proving the real analogue of the Fundamental Theorem of Tropical Geometry (Theorem~\ref{thm: multiplicity one cones are positive}).

\begin{lemma}\cite[Lemma 2.7]{Bossinger2017}
Let $I \subseteq \mathcal{C}[x_1^{\pm},\dots,x_n^{\pm}]$ be a prime ideal. If $w \in \Trop(V(I))$ is contained in a maximal cone of multiplicity one, then $\inw(I)$ has a unique toric ideal in its primary decomposition.
\end{lemma}

In Theorem~\ref{thm: weakly positive transverse intersection}, we provide a real analogue of the Transverse Intersection Theorem. We conclude this section by recalling the classical version of this theorem. By \cite[Definition 3.4.9]{maclagan2015introduction}, the tropical varieties $\Trop(V(I))$ and $ \Trop(V(J))$ \emph{meet transversally at $w \in \R^n$} if there exist polyhedral complexes
$\Sigma_1, \Sigma_2$ that are supported on $\Trop(V(I))$ and $ \Trop(V(J))$ such that the affine hull of the unique cells $\sigma_1 \in \Sigma_1$ and $\sigma_2 \in \Sigma_2$, that contain $w$ in their relative interior, is $\R^n$.

\begin{theorem}
\label{thm: weakly positive transverse intersection} \cite[Theorem 3.4.12]{maclagan2015introduction}
    Let $I, J \subseteq \mathcal{C}[x_1^\pm, \dots, x_n^\pm]$ be ideals defining the varieties $X = V(I)$, $Y = V(J)$. If $\Trop(X)$ and $ \Trop(Y)$  meet transversally at a point $w \in \R^n$ then $w $ lies in $ \Trop(X \cap Y )$.
\end{theorem}

\color{black}

\section{Positive tropicalization}
\label{Section:PositiveTrop}
The main result of this section is Theorem~\ref{thm: multiplicity one cones are positive}, which is a real analogue of the Fundamental Theorem (Theorem~\ref{Thm:FundamentalThm}). We begin with a brief review of positive tropicalization.
\color{black}
The tropicalization of the positive part of an algebraic variety has been studied in \cite{SpeyerWilliams::PosGrassmannian,ARDILA2006577,Vinzant::RealRadical,Tabera::Bases,Brandenburg::TropicalPositivity}, while the tropicalization of semi-algebraic sets has been investigated in \cite{Alessandrini::LogLimit,Allamigeon::TropSpecta,JellScheidererYu::RealTrop,Blekherman::Moments}. For a semi-algebraic set $X \subseteq \mathcal{R}^n$ we define its  \emph{positive tropicalization}
\begin{align}
\label{Def:PosTrop}
 \Trop^+(X) := \overline{ \val( X \cap \mathcal{R}_{>0}^n) },
 \end{align}
where the closure is taken in the Euclidean topology of $\mathbb{R}^n$.
Similar to \eqref{Eq::LogLimitBergman}, we define the \emph{positive logarithmic limit set} of $X \subseteq \mathbb{R}^n$ as $\LL^+(X ) := \LL(X \cap \mathbb{R}_{>0}^n)$. The following result is analogous to Theorem \ref{Thm::ComplexLogLimit}. 

\begin{theorem} 
\label{Thm::PosLogLimit}
\cite[Corollary 4.6]{Alessandrini::LogLimit} 
Let $X_\R \subseteq \mathbb{R}^n$ be a semi-algebraic set, and let $\mathcal{R}$ be a real closed  field  of rank one extending $\mathbb{R}$ with a non-trivial non-Archimedean valuation. Then
\[\Trop^+(X_\RR) = \LL^+(X_\R),\]
where $X_\RR \subseteq \mathcal{R}$ is the semi-algebraic set defined by the same polynomial inequalities as $X_\R$.
\end{theorem}

Justified by \Cref{Thm::PosLogLimit}, and with a slight abuse of notation, we define the positive tropicalization of a semialgebraic set $X_{\mathbb{R}} \subseteq \mathbb{R}^n$ as
\[ \Trop^+(X_\R) := \LL^+(X_\R).\]

Unlike tropicalization of algebraic varieties, positive tropicalization of an algebraic variety cannot be characterized using initial ideals. This makes positive tropicalization challenging to compute, even for algebraic varieties.
To overcome this difficulty, we 
rely on a relaxed notion that was introduced by Speyer and Williams in \cite{SpeyerWilliams::PosGrassmannian}
in order to compute the  tropicalization of the positive Grassmannian.
 For an algebraic variety $V \subseteq (\CC^*)^n$, its \emph{weakly positive tropicalization} is given by
\begin{align}
    \label{Def:WeakPosTrop}
    \Trop^+_\CC(V) := \overline{ \val( V \cap \mathcal{C}_{>0}^n) }.
\end{align}
For $\CC=  \mathbb{C}\{\!\{t \}\!\}$, the field of complex Puiseux series, $\Trop^+_\CC(V) $ has been called the positive part of the tropical variety in \cite{SpeyerWilliams::PosGrassmannian}. One can characterize the weakly positive tropicalization in terms of initial ideals as follows.

\begin{proposition}
\label{Prop::CharWeakSigned}
Let  $V \subseteq (\CC^*)^n$ be a very affine algebraic variety defined by an ideal $I \subseteq \mathcal{C}[x_1^\pm, \dots , x_n^\pm]$. For $w \in \mathbb{R}^n$ the following are equivalent
\begin{itemize}
    \item[(i)] $ w \in \Trop_\CC^+(V)$,
    \item[(ii)] $\exists u \in \mathbb{R}^n \colon V\big( \initial_u(\initial_w(I)) \big) \cap \widetilde{\mathcal{R}}^{n}_{>0} \neq \emptyset$,
    \item[(iii)] $\initial_w(I) \cap \widetilde{\mathcal{R}}_{\geq 0}[x_1, \dots , x_n] = \{ 0 \} $.
\end{itemize}
 Furthermore, the set of points $w \in \mathbb{R}^n$ with $V\big( \initial_w(I) \big) \cap \widetilde{\mathcal{R}}^{n}_{>0} \neq \emptyset $ forms a dense subset of $\Trop_\CC^+(V)$.
\end{proposition}
\begin{proof}
    The equivalence of $(i)$ and $(iii)$ holds by \cite[Proposition 2.2]{SpeyerWilliams::PosGrassmannian} over the Puiseux series $\CC = \mathbb{C}\{\!\{t \}\!\}$. The proof does however not rely on this particular choice of base-field.
    The equivalence $(ii) \iff (iii)$ follows from \cite{EinsiedlerTuncal::IdealPosPoly} (cf. \cite[Proposition 2.9]{Vinzant::RealRadical}), where again the only relevant property of the base-field is real-closedness.
    The density statement follows from the proof of \cite[Proposition 2.2]{SpeyerWilliams::PosGrassmannian}.
\end{proof}

The weakly positive tropicalization has an additional important property: there exists a polyhedral complex $\Sigma$ whose support is the complex tropicalization $\Trop(V(I))$, such that the weakly positive tropicalizaton $\Tropw(V(I))$ is the support of a closed subcomplex of $\Sigma$ \cite[Corollary 2.4]{SpeyerWilliams::PosGrassmannian}. This property might fail for the positive tropicalization $\Trop^+(V)$; see for example \cite[Figure 7]{Alessandrini::LogLimit}.

Since, by definition, positive elements are also weakly positive, it follows directly from \eqref{Def:PosTrop} and \eqref{Def:WeakPosTrop} that
\[\Trop^+(V) \subseteq \Trop_\CC^+(V).\]
This inclusion might be strict, as the following example demonstrates.

\begin{example}
\label{Ex:SingTrop}
We revisit the polynomial from Example~\ref{Ex::ComplexTrop}. Since 
\[ f(x) =  x_1^2-2x_1+1 + x_2 = (x_1 - 1)^2 + x_2 > 0,\]
for all $x \in \mathcal{R}_{>0}^2$, we have $\Trop^+(V(f)) = \emptyset$.

To determine the weakly positive tropicalization, we use Proposition~\ref{Prop::CharWeakSigned}. The initial form of $f$ with $w = (0,1)$  has been computed in \eqref{Ex:InitForm}. Since $V(\inw(f)) \cap \mathbb{R}_{>0} = V(x_1^2 - 2x_1 +x_1)  \cap \mathbb{R}_{>0} \neq \emptyset$, it follows that $ w \in \Trop_{\mathcal{C}}^+(V(f))$. In particular, the weakly positive tropicalization of $V(f)$ is non-empty. 
\end{example}

In Example~\ref{Ex:SingTrop}, the vanishing locus $V(\inw(f)) \cap \mathbb{R}^2_{>0}$ contained only singular points.  This led to $w = (0,1)$ being included in the weakly positive tropicalization, but not in the positive tropicalization of $V(f)$.

From \cite[Lemma 2.6]{Vinzant::RealRadical}, it follows that for an ideal $I$ defined by polynomials with coefficients in $\mathbb{R}$, the existence of a smooth point in $V(\inw(I)) \cap \mathbb{R}^n_{>0}$ implies $w \in \Trop^+(V(I))$. Using Hensel lifting, we generalize this result to the non-trivially valuated case.

\begin{lemma}
\label{Lemma: Hensel lifting}
    Let $X\subseteq (\CC^*)^n$ be a very affine variety and let $w \in \R^n$. If the scheme $\inw(X)$ contains a smooth point $\overline{x} \in \widetilde{\RR}^n_{>0}$, then $w$ lies in the positive tropicalization $\Trop^+(X)$.
\end{lemma}

To prove \Cref{Lemma: Hensel lifting} we need a technical result on initial degenerations.
Let $X$ be a very affine variety with defining ideal $I \subseteq \CC[x_1^\pm, \dots, x_n^\pm]$.
We denote by $\GG^n_{\CC^\circ} = \operatorname{spec}(\CC^\circ[x_1^\pm, \dots, x_n^\pm])$ the algebraic torus over the valuation ring, and by $\mathcal{X}$ the closure of $X$ in $\GG^n_{\CC^\circ}$.
This is the vanishing locus of $I \cap \CC^\circ[x_1^\pm, \dots, x_n^\pm]$. Furthermore, we denote by $X_s = \mathcal{X}\times_{\CC^\circ} \operatorname{spec}(\widetilde \CC)$ the special fiber.

\begin{example}
\label{Ex::special fiber}
    Consider the polynomial $g = x_1^2-2x_1+1 + tx_2  \in 
\mathbb{R}\{\!\{t \}\!\}[x_1,x_2]$, and let $X$ be the very affine variety cut out by $g$. The special fiber $X_s$ is the very affine variety defined by the polynomial $\widetilde{g} = \initial_{(0,0)}(g) =  x_1^2 - 2 x_1 + 1 \in \mathbb{R}[x_1,x_2]$.
\end{example}

\begin{lemma}
\label{prop: technical result on initial degs}
Let $X\subseteq (\CC^*)^n$ be a very affine variety defined by the ideal $I \subseteq \CC[x_1^\pm, \dots, x_n^\pm]$.
For $w = 0 \in \mathbb{R}^n$,
    the initial degeneration $\operatorname{in}_w(X)$ is equal to the special fiber $X_s$. Furthermore, the defining ideal $ I \cap \CC^\circ[x_1^\pm, \dots, x_n^\pm]$ of $\mathcal{X}$ is finitely generated.
\end{lemma}
\begin{proof}
    The first statement follows from Proposition 2.1.1 in the PhD thesis of David Speyer \cite{Speyer_Thesis}. We now show that $I \cap \CC^\circ[x_1^\pm, \dots, x_n^\pm]$ is finitely generated.
    To this end note that the quotient $\CC^\circ[x_1, \dots, x_n]/\left( I\cap \CC^\circ[x_1, \dots, x_n] \right)$ of the polynomial ring
     is torsion free, and hence flat over $\CC^\circ$.
    By a result from Raynaud and Gruson \cite[Corollary 3.4.7.]{flatification} every flat algebra of finite type over a domain is finitely presented, showing that the ideal $I\cap \CC^\circ[x_1, \dots, x_n]$ is finitely generated. In particular, the defining ideal $I\cap \CC^\circ[x_1^\pm, \dots, x_n^\pm]$ of $\mathcal{X}$ in the Laurent polynomial ring is finitely generated.
\end{proof}

\begin{proof}[Proof of \Cref{Lemma: Hensel lifting}]
    
    By \cite[Theorem 6.9]{JellScheidererYu::RealTrop} 
    the positive tropicalization $\Trop^+(X)$ does not change when we replace $\RR$ with a real closed, complete non-Archimedean valued field extension that has valuation group $\R$.
    From now on we assume without loss of generality
    that $\RR$ and $\CC$ have valuation group $\R$.
    Further, by replacing $X$ with $t^{-w}X$ we may assume without loss of generality that $w = 0$.  
    Consider as above the closure $\mathcal{X}$ of $X$ in the algebraic torus
    $\GG_{\CC^\circ}^n$ defined over the valuation ring.
    By \Cref{prop: technical result on initial degs}
     the special fiber $X_s = \mathcal{X}\times_{\CC^\circ}  \operatorname{spec}(\widetilde \CC)$ is equal to the initial degeneration $\inw(X)$.
 
    Let $U \subseteq \mathcal{X}$ be the open dense subset of smooth points. More concretely,
    by \Cref{prop: technical result on initial degs} there exists a finite choice of ideal generators $\langle f_1, \dots, f_s\rangle= I \cap \CC^\circ[x_1^\pm, \dots, x_n^\pm]$.
    For any such choice, consider the Jacobian $\operatorname{Jac}(f_1, \dots, f_s)$. Then $U$ is the complement of the closed subscheme of $\mathcal{X}$ defined by the vanishing of the maximal minors.
    By assumption, the residue of $\operatorname{Jac}(f_1, \dots, f_s)$ has full rank at $\overline{x}$, showing $\overline{x} \in U(\widetilde{\mathcal{R}})$.
    By construction of $U$, the natural morphism $ U \longrightarrow \operatorname{spec}(\RR^\circ)$ is smooth, hence by
    \cite[Lemma 29.36.20]{stacks-project}
    it factors through an \'etale map $\varphi: U \longrightarrow \mathbb{A}_{\RR^\circ}^k$ for some $k \geq 0$.
    We now choose any point $z \in (\RR^\circ)^k$ with residue $\varphi(\overline{x})$, defining a morphism $\tau \colon \operatorname{spec}(\RR^\circ) \longrightarrow \mathbb{A}_{\RR^\circ}^k $. 
    Let $Y$ be the fiber product $U \times_{\mathbb{A}_{\RR^\circ}^k} \operatorname{spec}(\RR^{\circ})$ with natural projections $\pi_1 \colon Y \longrightarrow U$, $\pi_2 \colon Y \longrightarrow \operatorname{spec}(\RR^\circ)$.
    \[\begin{tikzcd}
	U & \mathbb{A}_{\RR^\circ}^k \\
	Y & {\operatorname{spec}(\RR^\circ) }
	\arrow["\varphi", from=1-1, to=1-2]
	\arrow["\pi_1", from=2-1, to=1-1]
	\arrow["\pi_2",from=2-1, to=2-2]
	\arrow["\tau",from=2-2, to=1-2]
\end{tikzcd}\]
By construction as a fiber product, $Y$ is \'etale over $\operatorname{spec}(\RR^\circ)$.
Since $\RR^\circ$ is a complete local ring it is Henselian.
We employ Theorem 4.2 from \cite{milneLEC}, obtaining
a section $ \kappa :  \operatorname{spec}(\RR^\circ) \longrightarrow Y$ of $\pi_2$.
By construction, the resulting point $y = \pi_1  \circ \kappa$ in $X(\RR^\circ)$ has residue $\overline{x}$. In particular, its coordinates are positive and have valuation zero, finishing the proof.
\end{proof}

\begin{example}
\label{Example::SignedvsWeakSigned}
Consider the polynomial $g = x_1^2 - 2.1x_1 +1 +x_2$, which is obtained from the polynomial in Example~\ref{Ex:SingTrop} by perturbing its coefficients. Since for $w = (0,1)$
\[V(\inw(g)) \cap \mathbb{R}_{>0}^2 =  V(x_1^2 - 2.1x_1 +1) \cap \mathbb{R}_{>0}^2 \]
contains smooth points, Lemma~\ref{Lemma: Hensel lifting} implies that $w \in \Trop^+(V(g))$. In fact, we have
\[\Trop^+(V(g)) = \Trop^+_\mathcal{C}(V(g)) = \Cone\big( (0,1)\big).\]
Note that $\Cone\big( (0,1)\big)$ is a subfan of the inner normal fan of the Newton polytope of $g$ (cf. Figure~\ref{FIG1}(b)).
\end{example}

Using the results from this section, we now revisit the Fundamental Theorem of Tropical Algebraic Geometry and prove an analogous statement for positive tropicalization.

\begin{theorem}
\label{thm: multiplicity one cones are positive}
        Let $I \subseteq \mathcal{R}[x_1^\pm, \dots, x_n^\pm]$ be a  prime ideal, $X = V(I)$, and $w\in \mathbb{R}^n$ a simple point in the complex tropicalization $\Trop(X)$. Then the following are equivalent
        \begin{itemize}
    \item[(i)] $ w \in \Trop^+(X)$,
    \item[(ii)] $V\big( \initial_w(I) \big) \cap \widetilde{\mathcal{R}}^{n}_{>0} \neq \emptyset$,
    \item[(iii)] $\initial_w(I) \cap \widetilde{\mathcal{R}}_{\geq 0}[x_1, \dots , x_n] = \{ 0 \} $.
        \end{itemize}
        Furthermore, the initial degeneration $V\big( \initial_w(I) \big) \cap \widetilde{\mathcal{R}}^{n}_{>0}$ contains a smooth point.
\end{theorem}

\begin{proof}
The implications $"(i) \Rightarrow (iii)"$ and $"(ii) \Rightarrow (iii)"$ follow directly from \Cref{Prop::CharWeakSigned}. 
Note that we do not use the assumption that $w$ is a simple point here.

Assume that (iii) holds. By \Cref{Prop::CharWeakSigned}, there exists a weight vector $u \in \R^n$ such that the ideal $\initial_u(\initial_w(I))$ has a positive solution $x \in \widetilde{\mathcal{R}}_{>0}^n$.
We have that $\initial_u(\initial_w(I)) = \initial_{w + \epsilon u}(I)$ for all $\epsilon > 0$ small enough by \cite[Corollary 2.4.10]{maclagan2015introduction}. 
   Thus, we have $V\big( \initial_{w+\epsilon u}(I) \big) \cap \widetilde{\mathcal{R}}^{n}_{>0} \neq \emptyset$, which implies by Theorem~\ref{Thm:FundamentalThm} that $w + \epsilon u \in \Trop(X)$ for all small enough $\epsilon > 0$.

Let $\Sigma$ be the polyhedral complex with support $\Trop(X)$ as introduced at the end of Section~\ref{SubSec:TropAlg}.
Since $w$ is a simple point, it lies in the relative interior of a maximal cell $\sigma \in \Sigma$. Thus, for small enough $\epsilon > 0$ we have $w,w + \epsilon u \in \relint(\sigma)$  and therefore $\inw(I) = \initial_{w + \epsilon u}(I)$.
 In particular,  $V\big( \initial_{w}(I) \big) \cap \widetilde{\mathcal{R}}^{n}_{>0} = V\big( \initial_{w+\epsilon u}(I) \big) \cap \widetilde{\mathcal{R}}^{n}_{>0} \neq \emptyset$, which shows (ii).

    To finish the proof, we show that (iii) also implies (i).
    The multiplicity of $\Trop(X)$ at $w$ counts the components of $\inw(X) = V(\inw I)$ with appropriate multiplicities. Since $w$ is a simple point, the scheme $\inw(X)$ is irreducible and generically reduced.
    Together with \cite[Lemma 2.7]{Bossinger2017}, we see that the reduced-induced scheme is a toric variety. Its intersection with the positive orthant $\widetilde{\mathcal{R}}^{n}_{>0}$ is a dense subset
    and hence there is a smooth positive point $\overline{x}$ in $V(\inw (I))$.
    By \Cref{Lemma: Hensel lifting} we have $w \in  \Trop^+(X)$.
    \end{proof}

    \begin{remark}
    Consider the polynomial $f$ from Example~\ref{Ex::ComplexTrop}. By Remark~\ref{Rmk:NotSimplePoint}, $(0,1) \in \Trop(V(f))$ but it is not a simple point. 
    In Example~\ref{Ex:SingTrop}, we computed that $\Trop^+(V(f)) = \emptyset$ and $V(\initial_{(0,1)}(f)) \cap \mathbb{R}^2_{>0} \neq \emptyset$. Thus, this example shows that the implication $(ii) \Rightarrow (i)$ in Theorem~\ref{thm: multiplicity one cones are positive}  might fail for points that are non-simple.

    By \Cref{Prop::CharWeakSigned},  $V(\initial_{(0,1)}(f)) \cap \mathbb{R}^2_{>0} \neq \emptyset$ implies that  $\initial_{(0,1)}(\langle f\rangle) \cap \mathbb{R}_{\geq 0}[x_1,x_2] = \{ 0 \} $. Thus, the implication $(iii) \Rightarrow (i)$ might also fail for non-simple points.

     To see that $(iii) \Rightarrow (ii)$ might fail, consider the point $w=(0,0)$ and the same $f$ as above. We have 
     \[ \initial_{(0,0)}(f) = x_1^2-2x_1+1+x_2, \quad \initial_{(0,1)}( \initial_{(0,0)}(f)) = x_1^2-2x_1+1.\]
      Thus, using Proposition~\ref{Prop::CharWeakSigned}  we conclude that $w=(0,0)$ satisfies (iii) but not (ii).
    \end{remark}

\section{Computing positive tropicalizations}
\label{Sec::Comp}
In this section, we provide methods to compute the positive tropicalization of certain algebraic varieties. Our focus lies on the positive tropicalization of a linear space intersected with a toric variety that we will use in Section~\ref{Sec:PosSolsTrop} and~\ref{Section::SparsePolys}. Moreover, we prove a real analogue of the Transverse Intersection Theorem in Theorem~\ref{thm: weakly positive transverse intersection}.
 
\subsection{Linear spaces}
\label{Sec::LinearSpaces}
First, we consider linear spaces $V$ defined over $\mathbb{R}$. Since by Theorem \ref{Thm::PosLogLimit}, $\Trop^+(V)$ is independent of the choice of the extension field $\mathcal{R}$, in this section we use the field of formal real Puisseux series $\mathbb{R}\{\!\{t \}\!\}$, and its algebraic closure $\mathbb{C}\{\!\{t \}\!\}$.

From Lemma~\ref{Lemma: Hensel lifting}, it directly follows that the notions of positive and weakly positive tropicalization agree for linear spaces. In Proposition \ref{Prop:PositiveVSWeakPositiveTrop}, we give an elementary proof of this statement. 

\begin{proposition}
\label{Prop:PositiveVSWeakPositiveTrop}
Let $V \subseteq (\mathbb{C}\{\!\{t \}\!\}^*)^n$ be an algebraic variety defined by a linear ideal $I \subseteq \mathbb{R}\{\!\{t \}\!\}[x_1^\pm, \dots , x_n^\pm]$. Then $\Trop^+(V) =\Trop^+_\CC(V)$.
\end{proposition}

\begin{proof}
It suffices to show 
\[ \val \big( V \cap \mathbb{R}\{\!\{t \}\!\}^n_{>0} \big)  =  \val \big( V \cap \mathbb{C}\{\!\{t \}\!\}^n_{>0}\big).\]
 The inclusion $"\subseteq"$ follows immediately from \eqref{Eq:DiffPosIncl}. To prove the other inclusion, let $x(t) \in  V \cap \mathbb{C}\{\!\{t \}\!\}^n_{>0}$ and denote $\overline{ x(t)} = \big( \overline{x_1(t)} , \dots ,\overline{x_n(t)} \big)$ its complex conjugate, where each $\overline{x_i(t)}, \, i = 1 , \dots , n$ is obtained from $x_i(t)$ by taking the complex conjugate of its coefficients.
 
From $I \subseteq \mathbb{R}\{\!\{t \}\!\}[x_1^\pm, \dots , x_n^\pm]$, it follows that $\overline{ x(t)} \in V$. Since  $I$ is a linear ideal, we have that $x(t) + \overline{ x(t)} \in V $. For each $i = 1, \dots , n$, the leading coefficient $\lc(x_i(t))$ is real and positive. Thus, it follows that $x(t) + \overline{ x(t)} \in  \mathbb{R}\{\!\{t \}\!\}_{>0}^n$ and $\val(x(t)) = \val\big( x(t) + \overline{ x(t)} \big)$.
\end{proof}

It was shown by Sturmfels \cite{BerndSolving} that the logarithmic limit set $\LL(V)$ of a complex linear space only depends on the underlying matroid $\underline{M}$ of $V$, leading to the definition of the Bergman fan $\widetilde{B}(\underline{M})$ for any matroid $\underline{M}$.
The real case was studied in \cite{ARDILA2006577}, where the authors introduced the positive Bergman fan $\widetilde{B}^+(M)$ of an oriented matroid. If $M$ can be realized by a real linear space $V_\R\subseteq \R^n$, then
$\widetilde{B}^+(M)$ is shown to be equal to the weakly positive tropicalization $\Trop^+_\mathcal{C}(V_\R)$ \cite[Proposition 4.1]{ARDILA2006577}.

To state this result, we recall some notation and definitions from matroid theory.  We note that throughout we use the min-convention, deviating from \cite{ARDILA2006577}.
We will not recall the exact definition of an oriented matroid, we refer the reader to \cite[Chapter 6.3] {ziegler2012lectures}.

In our notation, an oriented matroid $M$ on the ground set $[n] = \{1, \dots, n\}$ is completely determined by a 
collection $\mathfrak{C}$ of pairs $C = (C^+, C^-)$, where $C^+, C^- \subseteq [n]$, $C^+ \cap C^- = \emptyset$, that satisfies certain axioms.
We call $C$ a signed circuit and $C^+$ and $ C^-$ its positive and negative part respectively.
The signed circuit $C$ has a support, denoted $\underline{C} = C^+ \cup C^-$.
The set of supports $\{\underline{C} \mid C \in \mathfrak{C}\}$ defines an unoriented matroid $\underline{M}$.
For every $w \in \R^n$ and circuit $C$ of $M$ we define the initial circuit
\begin{gather*}
    \inw(C) := (\inw(C)^+, \inw(C)^- ), \text{ where for } *\in \{+,-\} \text{ we denote }\\\inw(C)^* = \{ i \in C^* \mid w_i := \min_{j \in \underline{C}} w_j\}.
\end{gather*}
The set of circuits $\{\inw(C) \mid \ C \in \mathfrak{C}\}$ defines  an oriented matroid, denoted $M_w$ \cite[Proposition 2.3]{ARDILA2006577}.
The positive Bergman fan $\widetilde{B}^+(M)$ of $M$ is the set 
\begin{align*}
    \widetilde{B}^+(M) =\big\{ w \in \R^n \mid \ \forall C\in \mathfrak{C}\colon 
 \inw(C)^+ \neq \emptyset \text{ and } \inw(C)^- \neq \emptyset\big\}.
\end{align*}
The set $\widetilde{B}^+(M)$ can be given the structure of a polyhedral complex,  by identifying those $w, w' \in \widetilde{B}^+(M)$
for which the initial matroids $M_w, M_{w'}$ coincide.
This makes $\widetilde{B}^+(M)$ a polyhedral subcomplex of the Bergman fan
$\widetilde{B}(\underline{M})$, where the latter is endowed with its coarsest polyhedral structure. We call both polyhedral structures the \emph{coarse} structure on $\widetilde{B}^+(M)$ and $\widetilde{B}(\underline{M})$ respectively.

We now recall that 
to every linear space $V\subseteq \R^n$
one can associate an oriented matroid $M$ as follows. To every linear form $l = a_1x_{i_1}+\cdots + a_mx_{i_m}-\left( b_1x_{j_1}+\cdots+   b_kx_{j_k}
 \right) \in \mathbb{R}[x_1, \dots, x_n]$, vanishing on $V$, such that $a_i, b_j>0$ for all $i, j$, we associate the 
 pair $C = (C^+, C^-)$, where
 $C^+ = \{i_1, \dots, i_m\}, \ C^- = \{ j_1, \dots, j_k\}$.
The pairs $C$ for which the support $C^+ \cup C^-$ is minimal with respect to inclusion form the signed circuits of an oriented matroid $M$ and we say that $M$ is realized by the linear space $V$.
 \begin{example}
\label{Ex::BergmanFanofN} 
We consider the kernel $V$ of the matrix
  \[  N = \begin{pmatrix}
    -3 & 1 & -1 & -2 & 2 \\
    -1 & 1 & -1 & -1 & 1
\end{pmatrix}.\]
Then $V$ is spanned by the rows of the matrix
  \[
    \begin{pmatrix}
    v_1 & v_2 & v_3 & v_4 & v_5
    \end{pmatrix} =
  \begin{pmatrix}
    0 & 0 & 0 & 1 & 1 \\
    -1 & 1 & 0 & 2 & 0 \\
    0 & 1 & 1 & 0 & 0
\end{pmatrix},\]
and the inclusion minimal linearly dependent sets of column vectors are
\begin{align*}
    \big\{ &\{v_1,v_2,v_3\}, \phantom{-} \{v_1,v_4,v_5\}, \phantom{-} \{v_2,v_3,v_4,v_5\} \big\}.
\end{align*}
Each such set admits a unique (up to scaling) linear relation, for example
$0 = -v_2 +v_3 +2v_4 -2v_5$.
We associate with it a signed circuit $C$
with positive part $C^+ = \{3,4\}$ and negative part
$C^- = \{2,5\}$ and denote
$C = (\{3,4\}, \{2,5\})$.
Multiplying the linear relation by $-1$, we also have the signed circuit $(\{2,5\}, \{3,4\})$.
In total, the underlying oriented matroid $M$ of $V$ will comprise the following signed circuits:
  \begin{align*}
 \mathfrak{C} = \big\{ & (\{3\}, \{1,2\}),  \phantom{-}
 (\{5\}, \{1,4\}), \phantom{-}(\{2,5\}, \{3,4\}),
 \\ &(\{1,2\}, \{3\}),  \phantom{-}
 (\{1,4\}, \{5\}), \phantom{-}(\{3,4\}, \{2,5\}) \big\}.
 \end{align*}
For the weight vector $w = (0,2,0,2,0)$, we obtain the signed initial circuits
\begin{align*}
     \big\{ &(\{3\},\{1\}), \   (\{5\}, \{1\}), \ (\{5\}, \{3\}),   \\         
            &(\{1\}, \{3\}), \  (\{1\}, \{5\}), \ (\{3\}, \{5\}) \big\}.
\end{align*}
 In particular, $w$ lies in the positive Bergman fan $\widetilde{B}^+(M)$. 
\end{example}
 
The following proposition relates the positive Bergman fan $\widetilde{B}^+(M)$ to the logarithmic limit set $\LL^+(V)$ of $V$ and to the positive tropicalization $\Trop^+(V)$.
A related result was shown in \cite[Theorem 3.14]{Tabera::Bases}.
\begin{proposition}
Let $M$ be an oriented matroid that is realized by the linear space
$V\subseteq \R^n$. The following sets coincide:
\begin{itemize}
    \item[(i)] the positive tropicalization $\Trop^+(V)$,
     \item[(ii)] the logarithmic limit set $\LL^+(V)$,
    \item[(iii)] the positive Bergman fan $\widetilde{B}^+(M)$.
\end{itemize}
    \end{proposition}
\begin{proof}
Let $V_{\mathbb{C}\{\!\{t \}\!\}} \subseteq \mathbb{C}\{\!\{t \}\!\}^n$ be the Zariski closure of $V$ in $\mathbb{C}\{\!\{t \}\!\}^n$. From \cite[Proposition 4.1]{ARDILA2006577} it follows that the positive Bergman fan $\widetilde{B}^+(M)$ agrees with the weakly positive tropicalization $\Trop_\CC^+(V_{\mathbb{C}\{\!\{t \}\!\}})$. By Proposition \ref{Prop:PositiveVSWeakPositiveTrop}, we have 
\[\Trop^+_\CC(V_{\mathbb{C}\{\!\{t \}\!\}}) = \Trop^+(V_{\mathbb{C}\{\!\{t \}\!\}}). \]
By Theorem \ref{Thm::PosLogLimit}, the positive tropicalization of a variety defined over $\mathbb{R}$ is independent of the choice of base extension, and is equal to the logarithmic limit set. It follows that the sets (i)-(iii) coincide.
\end{proof}

Let $M$ be an arbitrary oriented matroid.
We now describe a procedure to compute 
the positive Bergman fan $\widetilde{B}^+(M)$ from the Bergman fan $\widetilde{B}(\underline{M})$ associated with the matroid $\underline{M}$. \Cref{algorithm} works by filtering the maximal cones of $\widetilde{B}(\underline{M})$.
The correctness of this approach follows from the observation that $\widetilde{B}^+(M)$, endowed with the coarse polyhedral structure, is a subfan of $\widetilde{B}(\underline{M})$.
Moreover, \cite[Corollary 5.3]{ARDILA2006577}  implies that $\widetilde{B}^+(M)$ is a pure subfan of dimension rank of $M$ if $\widetilde{B}^+(M)$ is nonempty.
An algorithm to compute $\widetilde{B}(\underline{M})$
is described in 
\cite{RinconComputingTropicalLinearSpaces}.
An implementation can be found in the software package $\texttt{Polymake}$ \cite{POLYMAKE}. 
The output of the algorithm is a simplicial polyhedral complex supported on $\widetilde{B}(\underline{M})$, called the cyclic Bergman fan.

\begin{algorithm*}
\caption{Computing maximal cones of $\widetilde{B}^+(M)$}\label{algorithm}
\begin{algorithmic}[1]
\State $\mathfrak{C} \gets \textit{ signed circuits of M }$
\State $\underline{M} \gets \textit{ underlying matroid of } M$
\State $\Sigma \gets \textit{ maximal cones of }\widetilde{B}(\underline{M})$
\State $\Sigma^+ \gets \emptyset$
\For{$\sigma \in \Sigma$}
    \State $w \gets \textit{ point in the relative interior of } \sigma$
    \State $IsIn\Sigma^+ \gets \operatorname{True}$
    \For{$C \in \mathfrak{C}$}
    \If {$\textit{length}( \inw (C)^+ ) = 0 
    \textbf{ or } \textit{length}( \inw(C)^- ) = 0$}
        \State $IsIn\Sigma^+ \gets \operatorname{False}$
    \EndIf
    \EndFor
            \If {$IsIn\Sigma^+$}
        \State $\Sigma^+ \gets \Sigma^+ \cup \{\sigma\}$
\EndIf
\EndFor
\State \textbf{return} $\Sigma^+$
\end{algorithmic}
\end{algorithm*}

\begin{example}
We again consider \Cref{Ex::BergmanFanofN},
and finally compute the positive Bergman fan of the oriented matroid $M$. 
Using $\texttt{Polymake}$~\cite{POLYMAKE} we compute the Bergman fan $\widetilde{B}(\underline{M})$.
Its rays are the vectors
\[
\{\rho_1, \dots, \rho_7\} = 
\left\{
 \begin{pmatrix}0\\ 1\\ 0\\ 0\\ 0 \end{pmatrix},
 \begin{pmatrix}0\\ 0\\ 0\\ 1\\ 0 \end{pmatrix},
 \begin{pmatrix}0\\ 0\\ 0\\ -1\\ -1 \end{pmatrix},
 \begin{pmatrix}0\\ -1\\ -1\\ 0\\ 0 \end{pmatrix},
 \begin{pmatrix}0\\ -1\\ -1\\ -1\\ -1 \end{pmatrix},
 \begin{pmatrix}0\\ 0\\ 0\\ 0\\ 1 \end{pmatrix},
 \begin{pmatrix}0\\ 0\\ 1\\ 0\\ 0 \end{pmatrix}
 \right\}.
\]
For each pair of indices $1 \leq i, j \leq 7$, we denote by $\operatorname{cone}(i,j) = \R \cdot \mathbf{1} + \R_{\geq 0}(\rho_i, \rho_j)$ the Minkowski sum of the span of the all ones vector, and of the positive hull of $\rho_i$ and $\rho_j$. Then the maximal cones of $\widetilde{B}(\underline{M})$ are

\begin{equation*}
  \begin{aligned}
\{\sigma_1, \dots, \sigma_{10}\} = 
&\{
\operatorname{cone}(1, 2), \
\operatorname{cone}(1, 3), \
\operatorname{cone}(2, 4), \
\operatorname{cone}(3, 5), \
\operatorname{cone}(4, 5), \\
&\operatorname{cone}(1, 6), \
\operatorname{cone}(4, 6), \
\operatorname{cone}(2, 7), \
\operatorname{cone}(3, 7), \
\operatorname{cone}(6, 7) \}.
  \end{aligned}
\end{equation*}
\end{example}
The positive Bergman fan $\widetilde{B}^+(M)$ is equal to the pure subfan of $\widetilde{B}(\underline{M})$ with maximal cones $\sigma_1,\, \sigma_2 ,\, \sigma_3 ,\, \sigma_4  ,\, \sigma_5$.

\begin{remark}
Since the initial submission of this paper, a version of Algorithm~\ref{algorithm} has been implemented by Yue Ren in the computer algebra system \texttt{Oscar.jl} \cite{OSCARbook}. For a linear ideal $I \subseteq \mathbb{R}[x_1, \dots, x_n]$, the function \texttt{positive\_tropical\_variety()} computes the positive Bergman fan $\widetilde{B}^+(M) =  \Trop^+(V(I))$.

In this implementation, instead of checking signed circuits, the code tests whether  $\inw(I)$ contains polynomials with only positive coefficients. By Proposition~\ref{Prop::CharWeakSigned}, this characterization is equivalent to $w \in \Tropw(V(I))=\Trop^+(V(I))$.
\end{remark}

\color{black}

\subsection{Functoriality of monomial maps}
\label{section: functoriality}

We study some natural functoriality properties of the weakly positive and the  positive tropicalization with respect to a monomial map. Let $A = (a_1 | \dots| a_r) \in \mathbb{Z}^{n \times r}$ be an integer matrix of rank $n \leq r$ inducing a finite monomial map
\begin{align}
\label{Eq:MonoMap}
    \varphi_A\colon (\CC^*)^n \longrightarrow (\CC^*)^r, \quad x \longmapsto \varphi_A(x) = \left(x^{a_1}, \dots, x^{a_r}\right).
\end{align}
The Zariski closure of the image of $\varphi_A$ in $\mathcal{C}^r$ is called an \emph{affine toric variety}, denoted by~$Y_A$. The following result is a generalization of \cite[Proposition 2.3]{TelenPos} to the field $\mathcal{C}$.

\begin{proposition}
\label{prop: positive part of toric vars}
For $Y_A =  \varphi_A \left( (\CC^*)^n \right)  \subseteq (\mathcal{C}^*)^r$ a very affine toric variety defined by a matrix $A \in \mathbb{Z}^{n \times r}$ of rank $n \leq r$, we have
\[ Y_A \cap \RR^r_{>0} = \varphi_A(\RR^n_{>0}).\]
\end{proposition}
\begin{proof}
The inclucion,  $Y_A \cap \RR^r_{>0} \supseteq \varphi_A(\RR^n_{>0})$ follows immediately. For the other inlcusion it is to show that $\operatorname{im}(\varphi_A)\cap \RR^r_{>0} \subseteq \varphi_A(\RR^n_{>0})$. To this end, denote $\operatorname{conj}\colon \CC \longrightarrow \CC$ the unique non-trivial $\RR$-algebra automorphism of $\CC$. For each $z \in \mathcal{C}^*$, the roots of the univariate polynomial $X^2- z \operatorname{conj}(z) \in \mathcal{R}[X]$ lie in $\mathcal{R}$. Therefore, $z \operatorname{conj}(z) \in \RR_{> 0}$. Define the norm of $z$ as the positive root of $z \operatorname{conj}(z)$, i.e. 
\[\left \| z \right \| := \sqrt{z \operatorname{conj}(z)} \in \RR_{>0}.\]

Let $(y_1,\dots y_r) = \varphi(x_1, \dots, x_n) \in \RR^r_{>0}$, where $x_1, \dots, x_n \in \CC^{*}$. Since taking the norm commutes with multiplication, we have 
\[(y_1, \dots, y_r) = (\left \| y_1 \right \|,\dots , \left \| y_r \right \|) = \varphi(\left \| x_1 \right \|, \dots, \left \| x_n \right \|) \in \varphi_A(\RR^n_{>0}),\]
finishing the proof.
\end{proof}

For the rest of this section, let $I \subseteq \CC[x_1^\pm, \dots, x_n^\pm]$ be a prime ideal defining a variety $X \subseteq \mathbb{G}^n_\CC$ and let
$J \subseteq \CC[x_1^\pm, \dots, x_r^\pm]$ be a prime ideal defining $Y \subseteq \mathbb{G}^r_\CC$
with $\varphi_A(X) \subseteq Y$.

\begin{lemma}
\label{lemma: induced map of initial degs}
For every weight vector $w \in \mathbb{R}^n$,
the monomial map $\varphi_A$ induces a monomial map on the initial degenerations
\[
\overline{\varphi_A}\colon \inw (X) \longrightarrow \initial_{A^{\top}w} (Y).
\]
Furthermore, if $X = \varphi_A^{-1}(Y)$, then
the schemes $\inw X $ and 
$\overline{\varphi_A}^{-1}(\initial_{A^\top w} (Y))$ agree.
\end{lemma}
\begin{proof}
We first treat the case $w = 0$.
Let $\mathcal{X}$ be the closure of $X$ in the algebraic torus $\mathbb{G}^n_{\CC^\circ} = \operatorname{spec}(\CC^\circ[x_1^\pm, \dots, x_n^\pm])$ and 
let $\mathcal{Y}$ be the closure of $Y$ in $\mathbb{G}^r_{\CC^\circ}$.
The initial degenerations $\initial_0(X)$ and $\initial_0(Y)$ are equal to the special fibers $X_s = \mathcal{X} \times_{\CC^\circ} \operatorname{spec}(\widetilde{\CC})$ and $Y_s = \mathcal{Y} \times_{\CC^\circ} \operatorname{spec}(\widetilde{\CC})$ respectively.
Note that the map $\varphi_A$ naturally extends to a map
\[\varphi_A\colon \operatorname{spec}(\CC^\circ[x_1^\pm, \dots, x_n^\pm]) \longrightarrow \operatorname{spec}(\CC^\circ[x_1^\pm, \dots, x_r^\pm])\] over the valuation ring.
It induces a morphism
\[\overline{\varphi_A}\colon \operatorname{spec}(\widetilde{\CC}[x_1^\pm, \dots, x_n^\pm]) \longrightarrow \operatorname{spec}(\widetilde{\CC}[x_1^\pm, \dots, x_r^\pm])\] on the special fibers.
By continuity of $\varphi_A$, the image of $\initial_0(X)$ is contained in $\initial_0(Y)$, showing the first part of statement.

We now assume $X = \varphi_A^{-1}(Y)$ and show the second part of the Lemma.
To this end, by \Cref{prop: technical result on initial degs} it suffices to show show the equality $\mathcal{X} = \varphi_A^{-1}(\mathcal{Y})$.
This follows from \cite[\href{https://stacks.math.columbia.edu/tag/0CMK}{Lemma 101.38.5}]{stacks-project}.
Alternatively, by flatness over $\CC^\circ$, to see that the natural scheme-theoretic inclusion
$\mathcal{X} \subseteq \varphi_A^{-1}(\mathcal{Y})$ is an equality, 
we may restrict to the generic fiber where equality follows by construction.
In order to reduce the case for a general choice of $w$ to the case $w = 0$ we note that  $\initial_0(t^{-w}X) =  \inw(X)$ and 
$\initial_0(t^{-A^\top w}Y) = \initial_{A^\top w} (Y)$, finishing the proof.
\end{proof}

Both the positive and the weakly positive tropicalization are functorial with respect to monomial maps. We will use this property to prove Corollary~\ref{Cor::PosTropMono} and \Cref{thm: weakly positive transverse intersection}.
\begin{proposition}
\label{prop: functoriality of monomial maps}
let $I \subseteq \CC[x_1^\pm, \dots, x_n^\pm]$ be a prime ideal defining a variety $X \subseteq \mathbb{G}^n_\CC$ and let
$J \subseteq \CC[x_1^\pm, \dots, x_r^\pm]$ be a prime ideal defining $Y \subseteq \mathbb{G}^r_\CC$
with $\varphi_A(X) \subseteq Y$.

    The linear images of the positive and of the weakly positive tropicalization of $X$ are contained in the respective tropicalizations of $Y$:
    \begin{align*}
            A^\top \Trop^+(X) \subseteq \Trop^+(Y), \qquad A^\top\Tropw(X) \subseteq \Tropw(Y).
    \end{align*}
    If $\varphi_A(X \cap \RR^n_{> 0 }) = Y\cap\RR^r_{> 0}$,
    then the first inclusion is an equality.
    Furthermore, if both $Y$ is contained in the image of $\varphi_A$,
    and $X = \varphi_A^{-1}(Y)$,
    then the second inclusion is also an equality.
\end{proposition}
\begin{proof}
    First we show the inclusion $A^\top \Trop^+(X) \subseteq \Trop^+(Y)$. The proof for the inclusion $A^\top \Tropw(X) \subseteq \Tropw(Y)$ is analogous.
    Let $x$ be any point in $X \cap \mathcal{R}^n_{>0}$ and let $w = \nu(x) \in \Trop^+(X)$. 
    Direct computation shows $\nu(\varphi_A(x)) = A^\top w$, giving the inclusion.

    We now assume $\varphi_A(X \cap \RR^n_{> 0 }) = Y\cap\RR^r_{> 0}$ and show the inclusion  $\Trop^+(Y) \subseteq A^\top\Trop^+(X)$.
    Let $y$ be in $Y \cap \RR^r_{> 0}$. By assumption there exists a point $x \in X \cap \RR^n_{>0}$ with $\varphi_A(x) = y$. Then $A^\top \nu(x) = \nu(y)$, finishing this part of the proof.

    We now assume $\varphi_A(X \cap \RR^n_{> 0 }) = Y\cap\RR^r_{> 0}$ 
    and $X = \varphi_A^{-1}(Y)$
    in order to finally show the desired equality $A^\top\Tropw(X) = \Tropw(Y)$ for the weakly positive tropicalization.
    Let $y \in Y, \ u  = \nu(y)$ such that
    $\initial_{u} (Y) $ contains a real point with positive entries.
    By \Cref{Prop::CharWeakSigned} these points form a dense subset of $\Tropw(\varphi_A(Y))$.
     Let $x \in \varphi_A^{-1}(y)$ and let $w = \nu(x)$.
    By \Cref{lemma: induced map of initial degs} the initial degeneration $\inw(X)$ is the preimage of $\initial_{u} (Y)$ under a monomial map. By \Cref{prop: positive part of toric vars} $\inw(X)$ contains a positive point, showing $w \in \Tropw(X)$.
    This shows the inclusion $\Tropw(Y) \subseteq A^\top\Tropw(X)$ and finishes the proof.
\end{proof}

As an immediate consequence of \Cref{prop: functoriality of monomial maps}, we obtain a description of the positive tropicalization of  toric varieties.
Let $A = (a_1 | \cdots | a_r) \in \mathbb{Z}^{n \times r}$ be a matrix of rank $n \leq r$, defining the very affine toric variety 
$Y_A = \{x^A = (x^{a_1}, \dots, x^{a_r}) \mid \ x \in (\mathcal{C}^*)^n\} \subseteq (\mathcal{C}^*)^r$.
\begin{corollary}
\label{Cor::PosTropMono}
    The positive tropicalization of $Y_A$ is the row span of $A$, that is
     \[\Trop^+\big( 
 Y_A \big) = \rowspan(A).\]
\end{corollary}
\begin{proof}
    This follows from \Cref{prop: positive part of toric vars} and \Cref{prop: functoriality of monomial maps} for $X = (\CC^*)^n$, $Y = Y_A$.

    Alternatively, observe that
    \[ \rowspan(A) \subseteq \Trop^+(Y_A) \subseteq \Trop(Y_A) = \rowspan(A),\]
    where the first inclusion follows from $\nu(\varphi_A(x)) = A^\top \nu(x)$ for $x \in \mathcal{R}^n_{>0}$. The second inclusion holds by definition. For the last equality, we refer to  \cite[Theorem 5.5.1]{maclagan2015introduction}.
\end{proof}

\subsection{Transverse intersections}

In this section, we prove a real analogue of the Transverse Intersection Theorem, enabling the efficient decomposition and computation of positive tropicalizations.
    Many algebraic varieties of interest can be expressed as intersections of simpler
    constituent varieties. Ideally, we would like to compute the tropicalizations of these complicated varieties through the tropicalization of their simpler components.
     For instance, if the tropicalization of a linear space and a toric variety satisfy a certain transversality condition, then the positive tropicalization of their intersection can be computed by determining the positive tropicalization of the linear space and the positive tropicalization of the toric variety, as discussed in Sections~\ref{Sec::LinearSpaces} and~\ref{section: functoriality}, and then intersecting the resulting two objects. Specifically, we require the following transversality assumption.

\begin{definition}
        Let $I,J \subseteq \mathcal{C}[x_1^\pm, \dots, x_n^\pm]$ be prime ideals. We say that the weakly positive tropicalizations \emph{$\Tropw(V(I))$ and $ \Tropw(V(J))$ meet transversally at $w \in \R^n$} if there exist polyhedral complexes
     $\Sigma_1, \Sigma_2$ that are supported on $\Tropw(V(I))$ and $ \Tropw(V(J))$ respectively with the following property: The affine hull of the unique cones
     $\sigma_1 \in \Sigma_1$ and $\sigma_2 \in \Sigma_2$, that contain $w$ in their relative interior, is $\R^n$.
\end{definition}

\begin{theorem}
\label{thm: weakly positive transverse intersection}
    Let $I, J \subseteq \mathcal{C}[x_1^\pm, \dots, x_n^\pm]$ be prime ideals defining the varieties $X = V(I)$, $Y = V(J)$. If  the weakly positive tropicalizations $\Tropw(X)$ and $ \Tropw(Y)$  meet transversally at  $w \in \R^n$,
    then $w $ lies in the weakly positive tropicalization $ \Tropw(X \cap Y )$.
    If in addition $w$ is a simple point of both complex tropicalizations $\Trop(X)$ and $ \Trop(Y)$, then
    $w$ lies in the positive tropicalization $\Trop^+(X \cap Y)$.
\end{theorem}

\begin{proof}
    For the first part of the proof we follow the proof of the Transverse Intersection Theorem \cite[Theorem 3.4.12]{maclagan2015introduction}, using the functoriality properties from \Cref{section: functoriality}.
    We fix maximal cells $\sigma_1, \sigma_2$ of $\Tropw(X)$ and $\Tropw(Y)$ respectively, containing $w$ in the relative interior. Let $w+L_1$ and $w+L_2$ be the affine hulls of $\sigma_1$ and $\sigma_2$ respectively, where $L_1$ and $L_2$ are linear spaces.
    By assumption we have $L_1+L_2 = \R^n$.
    We fix a linear basis $a_1, \dots, a_r \in \mathbb{Z}^n$  of $L_1 \cap L_2$ with integer entries, and extend it to a basis 
    $a_1, \dots, a_r, a_{r+1}, \dots,  a_{s}$ of $L_1$ and a basis
    $a_1, \dots, a_r, a_{s+1}, \dots, a_n$ of $L_2$ respectively. In particular $r = \operatorname{dim}(L_1 \cap L_2), \ s = \operatorname{dim}(L_1).$
    Then the integer matrix $A = (a_1| \dots |a_n) \in \mathbb{Z}^{n \times n}$
    induces a monomial map $\varphi_A \colon\mathbb{G}^n_\RR \longrightarrow \mathbb{G}^n_\RR, \ x \longmapsto (x^{a_i})_i$.

    We now denote $\widetilde{X} = \varphi_A^{-1}(X)$ with defining ideal $\widetilde{I} = \varphi_A^*(I)$ and $\widetilde{Y} = \varphi_A^{-1}(Y)$, $\widetilde{J} = \varphi_A^*(J)$.
    By \Cref{prop: functoriality of monomial maps}, we have $\Tropw(\widetilde{X}) = (A^\top)^{-1} \Tropw(X)$ and $ \Tropw(\widetilde{Y}) = (A^\top)^{-1} \Tropw(Y)$. Now $\Tropw(\widetilde{X})$ and $\Tropw(\widetilde{Y})$ intersect transversally at $\widetilde{w} = (A^\top)^{-1}w$ and we are left with showing that $\widetilde{w}$ is contained in $\Tropw(\widetilde X \cap \widetilde Y)$. For the rest of the proof we replace $I$ with $\widetilde{I}$, $J$ with $\widetilde{J}$ and $w$ with $\widetilde{w}$.

    By the proof of \cite[Theorem 3.3.8]{maclagan2015introduction}, we can find generators for $\initial_{ w} ( I)$ in the variables $x_{s+1}, \dots, x_{n}$. Similarly, we can find generators for $\initial_{ w}  (J)$ in the variables $x_{r+1}, \dots, x_s$.
    Since ${w}$ lies in $\Tropw( X)$ and $\Tropw( Y)$, by \Cref{Prop::CharWeakSigned} there exist weight vectors $u', u''$ such that the positive vanishing loci
    $V(\initial_{u'} \left( \initial_{ w}  (I)) \right) \cap \widetilde{\mathcal{R}}^n_{>0}$ and $V(\initial_{u''}  \left( \initial_{ w} ( J)) \right) \cap \widetilde{\mathcal{R}}^n_{>0}$ contain elements $z'$ and $z''$ respectively. Combining the entries of $u'$ and $u''$ yields the new weight vector ${u} = (u''_{1}, \dots, u''_{s}, u'_{s+1}, \dots, u'_{n})$. By homogeneity we have $ \initial_{u'} \left( \initial_{ w}  (I) \right) = \initial_{{u}}  \left( \initial_{ w}  (I) \right)$ and similarly 
    $\initial_{u''} \left( \initial_{ w}  (J) \right) = \initial_{{u}} \left( \initial_{ w} ( J) \right)$.
    Using the homogeneity of $\initial_{{w}}  (I)$ and $\initial_{{w}} ( J)$ we may apply an affine version of \cite[Lemma 3.4.10]{maclagan2015introduction} to obtain
    \[
    \initial_{{u}} \left( \initial_{{w}} ( I +  J) \right)
     =     \initial_{{u}}\left(\initial_{{w}} ( I ) \right)+ 
        \initial_{{u}}\left( \initial_{{w}} ( J) \right).
    \]
    By construction, the point $z = (z_1'', \dots, z_s'', z_{s+1}', \dots, z_n')$ lies in $V(\initial_{{u}} \left( \initial_{ w}( I +  J)) \right) \cap \widetilde{\mathcal{R}}^n_{>0}$, showing
    that $w$ is contained in $\Tropw({X} \cap {Y})$. This finishes the first part.
    
    For the second part we point out that, by
    \Cref{thm: multiplicity one cones are positive}, we may assume the point $z = (z_1'', \dots, z_s'', z_{s+1}', \dots, z_n')$ to be a smooth point in $V(\initial_{{u}} \left( \initial_{ w}( I )) \right) \cap \widetilde{\mathcal{R}}^n_{>0}$ and $V(\initial_{{u}} \left( \initial_{ w}( J ))\right) \cap \widetilde{\mathcal{R}}^n_{>0}$.
    Applying \Cref{Lemma: Hensel lifting} then finishes the proof. 
    Note that we use here that $\sigma_1$ and $\sigma_2$ have multiplicity one. Since the multiplicities of $\sigma_1$ and $\sigma_2$ might change when replacing $X$ with $\widetilde{X}$ and $Y$ with $\widetilde{Y}$, we need to give a final argument to show that $V(\initial_{\widetilde{u}}( \initial_{\widetilde{w}}( \widetilde{I} ) ) ) \cap \widetilde{\mathcal{R}}^n_{>0}$ and $V(\initial_{\widetilde{u}} ( \initial_{ \widetilde{w}}( \widetilde{J} ))) \cap \widetilde{\mathcal{R}}^n_{>0}$ also contain a smooth point, where $\widetilde{u} = (A^\top)^{-1}(u)$.
    By \Cref{lemma: induced map of initial degs}, we have $\initial_{\widetilde w+ \epsilon\widetilde{u}} (\widetilde{X}) = \varphi_A^{-1}(\initial_{w+\epsilon u} (X))$ and $\initial_{\widetilde w+\epsilon\widetilde{ u}} (\widetilde{Y}) = \varphi_A^{-1}(\initial_{w+\epsilon u} (Y))$. Since $\varphi_A$ defines a smooth map on the algebraic torus, the preimage of the open, dense smooth subsets of $\initial_{w+\epsilon u} (X)$ and $\initial_{w+\epsilon u} (Y)$ are again smooth and contain a positive point. Here we again identify $\initial_{\widetilde w+\epsilon \widetilde u} (\widetilde X)$ with $\initial_{\widetilde u} ( \initial_{\widetilde w} (\widetilde X) )$ and $\initial_{\widetilde w+\epsilon \widetilde u} \widetilde (Y)$ with $\initial_{\widetilde u} (\initial_{\widetilde w} (\widetilde Y))$. This finishes the proof of the second part.
\end{proof}

\section{Positive solutions via tropical geometry}
\label{Sec:PosSolsTrop}

Building on the results from the previous sections, we describe a method for studying the positive solutions of a polynomial equation system over a real closed field $\RR$ with non-trivial, non-Archimedean valuation. To this end, consider polynomials
\[f_i = \sum_{j = 1}^r c_{i,j}x^{a_j} \in \RR[x_1^\pm,\dots,x_n^\pm], \qquad i = 1, \dots , m. \]
We call $F = (f_1, \dots, f_m)$ a \emph{polynomial system} in $\RR[x_1^\pm, \dots, x_n^\pm]$, and  $C = \big( c_{i,j}\big) \in \mathcal{R}^{m \times r}$ the \emph{coefficient matrix} of $F$. The matrix $A \in \mathbb{Z}^{n \times r}$, whose columns are given by the exponent vectors $a_1, \dots , a_r$, is called the \emph{exponent matrix} of $F$. For convenience, we write the equations $f_1(x) = \dots =  f_m(x) =0$ in matrix form as
\begin{align}
\label{Eq::GenericSystemPuiss}
    C x^A = 0.
\end{align}
In Lemma~\ref{lemma: monomial map is bijective on positive part}, we relate the solutions of~\eqref{Eq::GenericSystemPuiss} to points in $\ker(C)$ intersected with the affine toric variety $Y_A$.

\begin{lemma}
\label{lemma: monomial map is bijective on positive part}
\label{Cor::AtrBijective}
Consider a polynomial system $F=(f_1, \dots , f_m)$ with coefficient matrix $C \in \mathcal{R}^{m\times r}$ and exponent matrix $A \in \mathbb{Z}^{n\times r}$. If $A$ has full rank $n \leq r$, then
\begin{itemize}
\item[(i)] the restriction of the monomial map $\varphi_A$ from~\eqref{Eq:MonoMap} bijectively identifies the positive vanishing locus $V(f_1, \dots, f_m)\cap \RR^n_{>0}$ with $\ker(C) \cap Y_A \cap \RR^n_{>0}$.
\item[(ii)] the linear map given by $A^\top$ induces a bijection between $\Trop^+(V(f_1, \dots, f_m))$ and $\Trop^+(\ker(C)\cap Y_A)$. 
\end{itemize}
\end{lemma}
\begin{proof}
   By \Cref{prop: positive part of toric vars}, we have $\varphi_A(\RR_{>0}^n) = Y_A \cap \RR^r_{>0}$. By~\eqref{Eq::GenericSystemPuiss}, the map 
    \[V(f_1, \dots, f_m)\cap \RR^n_{>0}  \longrightarrow  \ker(C) \cap Y_A \cap \RR^r_{>0}, \quad x \mapsto x^A \]
    is well defined. Since $A$ has full rank, it has a right inverse $B \in \mathbb{Q}^{r \times n}$, that is $AB = \Id_{n\times n}$. Since $\mathcal{R}$ is a real closed field, $z^B$ is uniquely defined for each $z \in \mathcal{R}^r_{>0}$. Thus, the map 
    \[\ker(C) \cap Y_A \cap \RR^r_{>0} \to  V(f_1, \dots, f_m)\cap \RR^n_{>0}, \quad z \mapsto z^B \]
     is well defined. A direct computation shows that $(x^A)^B = x^{AB}=x$, which completes the proof of (i). Part (ii) follows from (i) and \Cref{prop: functoriality of monomial maps}.
\end{proof}

\begin{remark}
The idea of relating solutions of polynomial equations to the intersection of a linear space with a toric variety is commonly used to study sparse equation systems, cf. \cite[Section 3.1.1]{sottile2011real} or \cite[Section 3.4]{telen2022introductiontoricgeometry}. While the restriction of the monomial map $\varphi_A$ from Lemma \ref{Cor::AtrBijective} is bijective, this might not hold for its complex counterpart. For example, for $A = \begin{pmatrix}
    2 & 0
\end{pmatrix} \in \mathbb{Z}^{1 \times 2}$ and $C = \begin{pmatrix}
    1 & -1
\end{pmatrix}  \in \mathcal{R}^{1 \times 2}$, the map 
\begin{align*}
\big\{ x \in \mathcal{C}^* \mid  x^2 -1  = 0 \big\} \to \ker(C) \cap Y_A, \quad  x \mapsto (x^2,1)
\end{align*}
is not injective, since $\varphi_A(1) = \varphi_A(-1) = (1,1)$.
\end{remark}

To compute $\Trop^+(\ker(C) \cap Y_A)$, we need to make the following assumption.

  \begin{definition}
 A polynomial system with exponent matrix $A \in \mathbb{Z}^{n \times r}$
and coefficient matrix $C \in \mathbb{\RR}^{m \times r}$ is called \emph{tropically transverse} if $\Trop_\mathcal{C}^+(\ker(C) )$ and  $\Trop_\mathcal{C}^+(Y_A)$ intersect transversally at every point $w \in \Trop_\mathcal{C}^+(\ker(C) ) \cap \Trop_\mathcal{C}^+(Y_A ) $.
\end{definition}

The following is the main result of this section.
\begin{theorem}
\label{thm: tropical equation solving}
Consider a polynomial system $F=(f_1, \dots , f_m)$ with coefficient matrix $C \in \mathcal{R}^{m\times r}$ and exponent matrix $A \in \mathbb{Z}^{n\times r}$ such that both $A$ and $C$ have rank $n$. If $F$ is tropically transverse, then the linear map $A^\top$ induces a bijection
    \[
    \Trop^+(V(f_1, \dots, f_m)) \longrightarrow \Trop^+(\ker(C)) \cap \operatorname{rowspan}(A).
    \]
\end{theorem}
\begin{proof}
From \Cref{lemma: monomial map is bijective on positive part}, it follows that $A^\top$ induces a bijection between $\Trop^+(V(f_1, \dots, f_m))$ and $\Trop^+(\ker(C) \cap Y_A)$. By Corollary \ref{Cor::PosTropMono}, we have $\Trop^+\big( 
 Y_A\big) = \rowspan(A)$. Thus, it is enough to show that
\begin{align}
\label{Eq:ProofTropTransverse}
\Trop^+(\ker(C) \cap Y_A) =  \Trop^+(\ker(C)) \cap \Trop^+(Y_A). 
\end{align}
The inclusion $\subseteq$ always holds. To show equality, we fix a point $w\in \Trop^+\big( \ker(C) \big) \cap \Trop^+\big( Y_A\big)$. By the transversality assumption, there exist polyhedral complexes $\Sigma_1, \, \Sigma_2$ supported on $\Trop_\mathcal{C}^+(\ker(C))$ and $\Trop_\mathcal{C}^+(Y_A)$ respectively, such that the affine hull of the unique cells $\sigma_1 \in \Sigma_1, \, \sigma_2 \in \Sigma_2$, which contain $w$ in their relative interiors, is $\mathbb{R}^r$. Since $\rk(A) = n$ and $\dim(\ker(C)) = r -n$, it follows that $\sigma_1$ and $\sigma_2$ are maximal cells in $\Sigma_1$ and $\Sigma_2$. Furthermore, since $\ker(C)$ is a linear subspace and $Y_A$ is a toric variety, both cells $\sigma_1$ and $\sigma_2$ have multiplicity $1$. Thus, $w$ is a simple point of both $\Trop(\ker(C))$ and $\Trop(Y_A)$.  Now, Theorem~\ref{thm: weakly positive transverse intersection} implies that $w \in \Trop^+(\ker(C) \cap Y_A)$. This shows~\eqref{Eq:ProofTropTransverse} and completes the proof.
\end{proof}

We conclude this section with a corollary of Theorem~\ref{thm: tropical equation solving}, which allows us to provide lower bounds on the positive real solutions of a specific type of equation system.

\begin{corollary}
\label{cor: generalized Viro patchworking}
     Let $\RR = H(\R_{\text{an}^*})$ be the Hardy field and let $F = (f_1, \dots, f_m)$ be a polynomial system with coefficient matrix $C \in \mathcal{R}^{m\times r}$ and exponent matrix $A \in \mathbb{Z}^{n\times r}$ such that both $A$ and $C$ have rank $n$.

    There exists $\varepsilon \in \mathbb{R}_{>0}$ such that for all $\tau \in (0,\varepsilon)$ the number of positive real solutions of the induced system
    $f_{1,\tau}(x) = \dots = f_{n,\tau}(x) = 0$ is bounded from below by the cardinality of $\Trop^+\big( \ker(C)  \cap Y_A\big).$
    Furthermore, if $F$ is tropically transverse, this lower bound equals the cardinality of
    $\Trop^+(\ker(C))  \cap \operatorname{rowspan}(A).$
\end{corollary}
 \begin{proof}
     This follows from \Cref{Thm::NumOfSols}, \Cref{lemma: monomial map is bijective on positive part} and \Cref{thm: tropical equation solving}.
 \end{proof}

\begin{example}
    Consider the univariate polynomial  
    \[  f=  x^2 -(2+t)x+1+t = \big(x-1\big)\big(x-(1+t)\big). \]
    Its coefficient and exponent matrices are given by
      \begin{align*}
          C = \begin{pmatrix}
        1 & -(2+t) & 1+t
    \end{pmatrix}, \qquad  A = \begin{pmatrix}
        2 & 1 & 0
    \end{pmatrix}.
      \end{align*}
   
      There are two solutions of $f(x) = 0$ in $\mathcal{R}_{>0}$, namely $1$ and $1+t$. Consequently, by Theorem~\ref{Thm::NumOfSols}, there exists $\varepsilon \in \mathbb{R}_{>0}$ such that for all $\tau \in (0,\varepsilon)$ the real polynomial
          \[  f_\tau=  x^2 -(2+\tau)x+1+\tau = \big(x-1\big)\big(x-(1+\tau)\big). \]
      has two positive real solutions. Since $\nu(1) = \nu(1+t) = 0$, $\Trop^+(V(f))$ and consequently  $\Trop^+(\ker(C) \cap Y_A)$ contain exactly one point. This example demonstrates that in some cases the bound from \Cref{cor: generalized Viro patchworking} is not an equality.
\end{example}

\section{Applications of the tropical bound}
\label{Section::SparsePolys}

\subsection{Vertically parametrized polynomial systems}

In this section, we provide a lower bound on the number of positive real solutions, for sufficiently small values of $t > 0$, for an equation system of the form 
\begin{align}
\label{Eq::BihanViroSystem}
 N \diag(t^h) x^A = 0,
\end{align}
 where $A  \in \Z^{n \times r},\, N \in \mathbb{R}^{n \times r}$ are matrices of full rank $n$, and $h \in \mathbb{R}^r$. The coefficient matrix of the polynomial system in~\eqref{Eq::BihanViroSystem} equals $C = N \diag(t^h) $, that is, the $j$-th column of $N$ is scaled by $t^{h_j}$. These types of parametrized equation system are called \emph{vertically parametrized} in the literature \cite{RenHelminck,FeliuDim,HelminckHenrikssonRen}. They appear as the steady states equations of chemical reaction networks \cite{DickensteinInvitation,FeliuRoleofAlg}.
While in the previous sections $t$ always denoted a formal variable in the appropriate Puiseux series fields, in this section we consider $t$ mostly as a small positive real number.

\begin{theorem}
\label{Prop::LowerBoundVertic}
      For generic $h \in \mathbb{R}^r$, the polynomial system in \eqref{Eq::BihanViroSystem} is tropically transverse. If \eqref{Eq::BihanViroSystem} is tropically transverse, there exists $\varepsilon \in \mathbb{R}_{>0}$ such that for all $t \in (0,\varepsilon)$ the number of positive real solutions of 
  \[ N \diag(t^h) x^A = 0\]
  is at least the number of points in $\big(\Trop^+(\ker(N)) - h \big) \cap \rowspan(A)$.
\end{theorem}

\begin{proof}First, we view $t^{h_j}, \, j = 1, \dots, r$ as generalized real Puiseux series, and consider the coefficient matrix $N \diag(t^h)$ as a matrix with entries in the Hardy field $H(\mathbb{R}_{\text{an}^*})$. A direct computation shows that 
    \[\Trop^+\big(\ker(N\diag(t^h)) \big)  = \Trop^+(\ker(N)) - h.\]
    By a similar argument as in \cite[Proposition 3.6.12]{maclagan2015introduction}, it follows that for a generic choice of $h$, the intersection $\big(\Trop_\mathcal{C}^+(\ker(N)) - h \big) \cap \rowspan(A)$ is transverse. Now, the result follows directly from Corollary~\ref{cor: generalized Viro patchworking}.
\end{proof}

\begin{example}
\label{Ex::VertParamBound}
Consider the following vertically parametrized polynomial system
\begin{equation}
\label{Eq::ExSystem}
    \begin{aligned}
    -3t^{h_1}  + t^{h_2} x_1^2  - t^{h_3} x_2^2 - 2t^{h_4} x_1^2 x_2^2+2t^{h_5} x_1 x_2 = 0, \\
     -\phantom{3}t^{h_1}  + t^{h_2} x_1^2  - t^{h_3} x_2^2 -  \phantom{2}t^{h_4} x_1^2 x_2^2+\phantom{2}t^{h_5}x_1 x_2 = 0,
\end{aligned}
\end{equation}
where $x_1,x_2$ are the variables and $t$ is a positive real parameter. The corresponding coefficient and exponent matrices are
  \[  N = \begin{pmatrix}
    -3 & 1 & -1 & -2 & 2 \\
    -1 & 1 & -1 & -1 & 1
\end{pmatrix}, \quad   A =  \begin{pmatrix}
         0 & 2 & 0 & 2 & 1\\
         0 & 0 & 2 & 2 & 1
    \end{pmatrix}. \]

    As in Example~\ref{Ex::BergmanFanofN}, one computes that the positive tropicalization $\Trop^+(\ker(N))$ contains the vectors 
    \begin{align}
    \label{Eq:ExVector}
    (0,2,0,2,0), \quad \text{and} \quad (0,-1,-1,-2,-2).
    \end{align}
    For $h = (0,0,0,0,-1)$, the intersection $\big(\Trop_\mathcal{C}^+(\ker(N)) - h \big) \cap \rowspan(A)$ is transverse. Thus, $N\diag(t^h)x^A=0$ is a tropically transverse system and Theorem~\ref{Prop::LowerBoundVertic} applies.

   The translation of the vectors in \eqref{Eq:ExVector} by the vector $-h = (0,0,0,0,1)$ lie in the row span of $A$. Using \texttt{Polymake}~\cite{POLYMAKE}, one can compute that these are the only points in the intersection $\big(\Trop^+(\ker(N)) - h \big) \cap \rowspan(A)$. Thus, for small enough $t > 0$ the equation system~\eqref{Eq::ExSystem} for $h = (0,0,0,0,-1)$ has at least two positive real solutions by Theorem~\ref{Prop::LowerBoundVertic}.
\end{example}

\begin{remark}
\label{Remark:ViroforVertParam}
Theorem~\ref{Prop::LowerBoundVertic} is closely related to the classical Viro's method~\cite{Viro_Dissertation,Bernd::Patchworking}. However, Viro's method for complete intersections \cite[Theorem 4]{Bernd::Patchworking} cannot be directly applied to vertically parametrized systems~\eqref{Eq::BihanViroSystem}. In a vertically parametrized system, a monomial~$x^{a_j}$ must be multiplied by the same $t^{h_j}$ in each polynomial of the system, which can result in the height function not being sufficiently generic in the sense of \cite{Bernd::Patchworking}. For example, the height functions for the two polynomials in~\eqref{Eq::ExSystem} are identical.
\end{remark}

\subsection{Positively decorated simplices}

In the following, we compare Theorem~\ref{Prop::LowerBoundVertic} to a method based on \emph{positively decorated simplices} from \cite{PolyhedralMethodSparsePositive}, which also provides a lower bound on the number of positive real solutions of~\eqref{Eq::BihanViroSystem} for sufficiently small values of $t>0$. Throughout this section, we assume that the columns of the exponent matrix $A$ are pairwise distinct. We denote these columns by $\alpha_1, \dots, \alpha_r$. Each $h \in \mathbb{R}^r$ induces a regular subdivision $\Gamma_h$ of $\alpha_1, \dots, \alpha_r$, whose \emph{cells} $\Delta \in  \Gamma_h$ are given by
\begin{align}
\label{Eq::DefSimplexMin}
     \Delta = \big\{ \alpha_k \mid \exists v \in \mathbb{R}^n\colon (v,1) \cdot (\alpha_k,h_k) = \min_{j \in [r]} (v,1) \cdot (\alpha_j,h_j) \big\}. 
\end{align}
If the cardinality of each $\Delta \in \Gamma_{h}$ is exactly $\dim \Conv(\Delta ) +1$, then $\Gamma_h$  is a triangulation. Geometrically, we can think about the elements of a cell $\Delta \in \Gamma_h$ as follows. Consider the convex hull $\hat{P}$ of the \emph{lifted points} $\hat{\mathcal{A}} = \{ \hat{\alpha}_1, \dots ,\hat{\alpha}_r\}= \{(\alpha_1,h_1), \dots ,(\alpha_r,h_r) \}$. A face of $\hat{P}$ is called a \emph{lower face} if the last coordinate of one of its inner normal vectors is positive. Thus, $\{\alpha_{i_1}, \dots , \alpha_{i_k}\}$ is a cell in $\Gamma_h$ if and only if there exists a lower face $\hat{F}$ of $\hat{P}$ with inner normal vector $(v,1) \in \mathbb{R}^{n+1}$ such that $\hat{F} \cap \hat{\mathcal{A}} = \{\hat{\alpha}_{i_1}, \dots , \hat{\alpha}_{i_k} \}$.

\begin{example}
\label{Ex::SubDivision}
    Consider the exponent matrix $A$ from Example~\ref{Ex::VertParamBound}. For $h = (0,0,0,0,-1)$, the regular subdivision $\Gamma_h$ has four  $2$-dimensional cells
    \begin{align*}
\Delta_1 = \Big\{ \begin{pmatrix}
        0\\
        0
    \end{pmatrix} ,\begin{pmatrix}
        2\\
        0
    \end{pmatrix}, \begin{pmatrix}
        1\\
        1
    \end{pmatrix}\Big\}, \quad \Delta_2 = \Big\{ \begin{pmatrix}
        0\\
        0
    \end{pmatrix} ,\begin{pmatrix}
        0\\
        2
    \end{pmatrix}, \begin{pmatrix}
        1\\
        1
    \end{pmatrix}\Big\}, \\
    \Delta_3 = \Big\{ \begin{pmatrix}
        2\\
        0
    \end{pmatrix} ,\begin{pmatrix}
        2\\
        2
    \end{pmatrix}, \begin{pmatrix}
        1\\
        1
    \end{pmatrix}\Big\}, \quad \Delta_4 = \Big\{ \begin{pmatrix}
        0\\
        2
    \end{pmatrix} ,\begin{pmatrix}
        2\\
        2
    \end{pmatrix}, \begin{pmatrix}
        1\\
        1
    \end{pmatrix}\Big\},
    \end{align*}
    which are depicted in Figure~\ref{FIG2}(a). In this example, the lifted points are
    \[ \hat{\alpha}_1 = \begin{pmatrix}0\\0\\0\end{pmatrix}, \quad \hat{\alpha}_2 = \begin{pmatrix}2\\0\\0\end{pmatrix}, \quad \hat{\alpha}_3 = \begin{pmatrix}0\\2\\0\end{pmatrix}, \quad \hat{\alpha}_4 = \begin{pmatrix}2\\2\\0\end{pmatrix}, \quad \hat{\alpha}_5 = \begin{pmatrix}1\\1\\-1\end{pmatrix}.\]
    For the cell $\Delta_2$, an inner normal vector of the corresponding lower face of $\hat{P}$ is $(v,1) = (1,0,1)$. For an illustration, we refer to Figure~\ref{FIG2}(b).
\end{example}

\begin{figure}[t]
\centering
\begin{minipage}[h]{0.3\textwidth}
\centering
\includegraphics[scale=0.4]{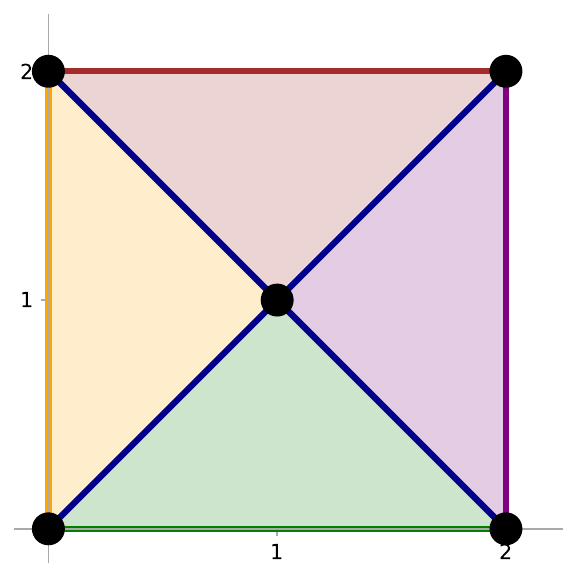}
{\small (a)}
\end{minipage}
\hspace{40pt}
\begin{minipage}[h]{0.3\textwidth}
\centering
\includegraphics[scale=0.4]{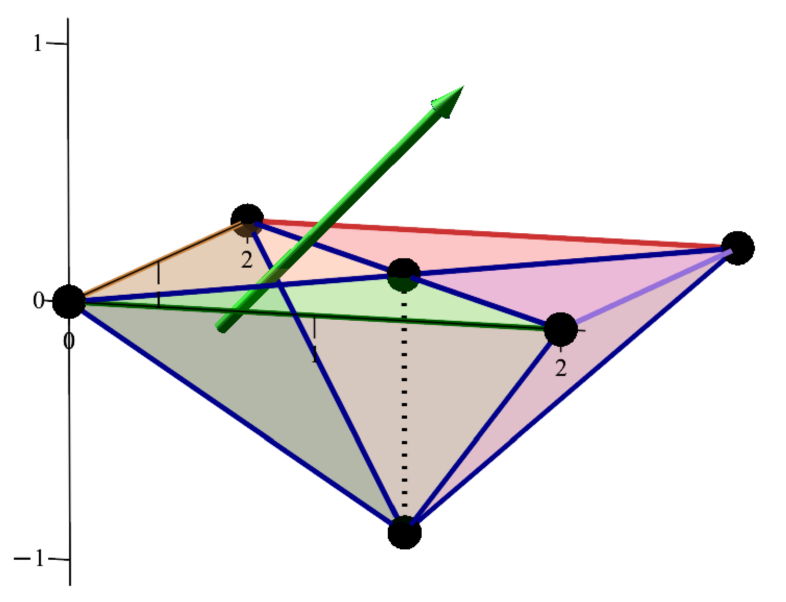}
{\small (b)}
\end{minipage}

\caption{{\small  Regular subdivision and the lifted points from Example~\ref{Ex::SubDivision} }}\label{FIG2}
\end{figure}

Now, recall the definition of a positively decorated simplex from \cite{PolyhedralMethodSparsePositive}. For an $n$-dimensional cell $\Delta = \{\alpha_{i_1}, \dots \alpha_{i_{n+1}}\} \in \Gamma_h$ with $i_1 < \dots < i_{n+1}$, we denote by $N_\Delta$ the matrix which is obtained from $N$ by selecting the columns with indices $i_1, \dots ,i_{n+1}$. We say that the matrix $N$ \emph{positively decorates the simplex} $\Delta \in \Gamma_h$ if the kernel of the submatrix $N_\Delta$ is one dimensional and intersects the positive orthant $\mathbb{R}^{n+1}_{>0}$. The importance of this notion is due to the following lemma.

\begin{lemma} \cite[Proposition 3.3]{PolyhedralMethodSparsePositive}
\label{Lemma::SubSolution}
Let  $A  \in \Z^{n \times r},\, N \in \mathbb{R}^{n \times r}$ be matrices of rank $n$ and $h \in \mathbb{R}^r$. If $N$ positively decorates the $n$-simplex $\Delta \in \Gamma_h$, then $N_\Delta x^{A_\Delta} = 0$ has exactly one positive real solution. Moreover, this unique positive solution is non-degenerate.  
\end{lemma}

Using Lemma~\ref{Lemma::SubSolution}, the authors in \cite{PolyhedralMethodSparsePositive} gave the following lower bound on the positive real solutions of the system \eqref{Eq::BihanViroSystem}.

\begin{theorem}
\label{Thm:Bihan}
Let  $A  \in \Z^{n \times r},\, N \in \mathbb{R}^{n \times r}$be matrices of rank $n$ and $h \in \mathbb{R}^r$.
Then there exists $\varepsilon \in \mathbb{R}_{>0}$ such that for all $t \in (0,\varepsilon)$ the number of positive real solutions of the system \[N \diag(t^h) x^A = 0\]
   is at least the number of positively decorated simplices of $\Gamma_h$.
\end{theorem}

\begin{example}
\label{Ex:BetterBound}
    We revisit the equation system from Example~\ref{Ex::VertParamBound} and~\ref{Ex::SubDivision}. The kernel of the matrix 
    \[N_{\Delta_2} = \begin{pmatrix}
        -3 & -1 & 2\\
        -1 & -1 & 1
    \end{pmatrix}\]
    is spanned by the vector $(1,1,2)$. Thus, $\Delta_2$ is a positively decorated $2$-simplex. A simple computation shows that this is the only positively decorated $2$-simplex in $\Gamma_h$. Thus, Theorem~\ref{Thm:Bihan} gives that the equation system  \eqref{Eq::BihanViroSystem}, for this choice of $N,A$ and $h$, has at least one positive real solution for small values of $t > 0$. 
    
    Note that using Theorem~\ref{Prop::LowerBoundVertic}, we showed  in Example~\ref{Ex::VertParamBound} that this system has at least two positive real solutions for sufficiently small $t > 0$.
\end{example}

In Example \ref{Ex:BetterBound}, we presented a case where Theorem~\ref{Prop::LowerBoundVertic} offers a strictly larger lower bound than Theorem~\ref{Thm:Bihan}. We conclude this section with a result showing that the bound given by Theorem~\ref{Prop::LowerBoundVertic} is always at least as good as the one provided by Theorem~\ref{Thm:Bihan}.
For technical reasons, we assume that $h$ has rational coordinates in order to prove this statement.

\begin{proposition}
\label{Prop:InjMapPosDec}
Let  $A  \in \Z^{n \times r},\, N \in \mathbb{R}^{n \times r}$ be matrices of rank $n$ and $h \in \mathbb{Q}^r$. Let $\mathcal{R} = H(\mathbb{R}_{\text{an}^*})$ denote the Hardy field and consider $N \diag(t^h) \in \mathcal{R}^{n \times r}$. The map 
      \begin{align*}
          \Big\{ \text{positively decorated } n\text{-simplex in } \Gamma_h \Big\} &\to \Trop^+\Big(\ker\big( N \diag(t^h)  \big)\cap Y_A \Big) \\
        \Delta \quad &\mapsto \quad A^\top v,
      \end{align*}
     is injective where $v \in \mathbb{R}^n$ is the unique vector satisfying
      \[
\Delta = \big\{ \alpha_k \mid  (v,1) \cdot (\alpha_k,h_k) = \min_{j \in [r]} (v,1) \cdot (\alpha_j,h_j) \big\} .\]
\end{proposition}

\begin{proof}
   We follow the proof of \cite[Theorem 3.4]{PolyhedralMethodSparsePositive}. Since $\Conv( (\alpha_k, h_k) \mid \alpha_k \in \Delta )$ is a polytope with rational vertices, the inner normal vector $(v,1)$ satisfies $v \in \mathbb{Q}^n$. Choose $\lambda \in \mathbb{Z}_{>0}$ such that $\lambda \big( ( v,1)\cdot (\alpha_j,  h_j) \big)\in \mathbb{Z}$ for all $j \in [r]$ and consider the polynomial 
    \[ f_{i,t}(x) := \sum_{j\in [r]} c_{ij}t^{\lambda  h_j}x^{\alpha_j}.\]
    Define $\beta :=  \min_{j \in [r]}  (\lambda v,\lambda) \cdot (\alpha_j, h_j) $. From \eqref{Eq::DefSimplexMin}, it follows that
    \[F_i(x,t) := t^{-\beta} f_{i,t}(t^{\lambda v} \ast x) = \sum_{j \in \Delta } c_{ij}x^{\alpha_j} + \sum_{j \in [r] \setminus \Delta } c_{ij}t^{(\lambda v,\lambda)\cdot (\alpha_j,  h_j) - \beta} x^{\alpha_j}. \]
Since the exponent of $t$ in the second sum is a strictly positive integer, the function
\[ F\colon \mathbb{R}^n_{>0} \times \mathbb{R} \to \mathbb{R}^n, \quad (x,t) \mapsto \big(F_1(x,t),\dots,F_n(x,t)\big)\]
is polynomial in $t$. In particular, $F$ is a real-analytic function on $\mathbb{R}^n_{>0} \times \mathbb{R}$. Since $\Delta$ is positively decorated, from Lemma~\ref{Lemma::SubSolution}, it follows that $F(x,0) = 0$ has a unique positive solution $z \in \mathbb{R}^n_{>0}$ such that the Jacobian matrix $\big(\tfrac{\partial F_j}{\partial x_i}(z,0)\big)_{i,j}$ is invertible. By the Real Analytic Implicit Function Theorem \cite[Theorem 1.8.3]{krantz2002primer}, there exists an open interval $(-\epsilon,\epsilon )\subseteq \mathbb{R}$ and a real analytic function $g\colon (-\epsilon,\epsilon ) \to \mathbb{R}^n$ such that $g(0) = z$, and $F(g(t),t) = 0$ holds for all $t \in (-\epsilon,\epsilon )$.  Consider the coordinate change
\[ \psi\colon [0,\epsilon^{\lambda}) \to [0,\epsilon), \quad t \mapsto t^{\tfrac{1}{\lambda}}. \]
Thus, for  $t \in (0,\epsilon^{\tfrac{1}{\lambda}})$, we have $F(g(\psi(t)),\psi(t)) = 0$. Since $g$ is real analytic, for all $i \in [n]$, we have
    \[ g_i(\psi(s)) = g_i(0) + \tfrac{g_i'(0)}{1!} s^{\tfrac{1}{\lambda}} + \tfrac{g_i''(0)}{2!}s^{\tfrac{2}{\lambda}}+ \dots ,\]
    cf. \cite[Corollary 1.1.10]{krantz2002primer}, which is an absolutely locally convergent generalized real Puiseux series around $0$. Since $g_i(\psi(0)) = z_i > 0$, it follows that $[t^{ v_i} g_i(\psi(t)) ]>0$ and $\val([t^{ v_i} g_i(\psi(t) ]) = v_i$ for $i \in[n]$.  
    
From \cite[Proposition 5.9]{Coste::oMinimal}, it follows that $[t^{v} \ast g(\psi(t))]$ is a solution of $N \diag( t^{ h}) x^A = 0$ in $\mathcal{R}_{>0}^n$. In particular
  \[  v = \val([t^{v } \ast g(\psi(t))]) \in   \Trop^+\big( \big\{ x \in \mathcal{R}_{>0}^n \mid N \diag( t^{ h}) x^A = 0 \big\} \big) \]
    Now, the statement follows from Lemma \ref{Cor::AtrBijective}(ii). 
\end{proof}

\begin{example}
\label{Eq:StircExample}
Consider again our running example. For $h=(0,0,0,0,-1)$, we computed in Example~\ref{Ex::SubDivision} and \ref{Ex:BetterBound} that $\Delta_2$ is the only positively decorated simplex and the corresponding inner normal vector is given by $(v,1)=(1,0,1)$. From Proposition~\ref{Prop:InjMapPosDec}, it follows that  $A^\top(1,0) = (0,2,0,2,1) \in \big(\Trop^+(\ker(N)) - h \big) \cap \rowspan(A)$. Note that in Example~\ref{Ex::VertParamBound} we showed that this tropical intersection contains an additional point  $(0,-1,-1,-2,-1)$.
\end{example}

\subsection{Chemical reaction networks}
\label{Sec:CRNT}
In this section, we demonstrate how \Cref{cor: generalized Viro patchworking} can be used to analyze the number of positive steady states of a reaction network. For a comprehensive overview of reaction networks, we refer to~\cite{DickensteinInvitation,FeliuRoleofAlg}. Here, we consider the hybrid histidine kinase network, which has been studied in~\cite{HHK,PLOS_IdParaRegions}.
\begin{align*}
\mathrm{HK}_{00} \xrightarrow{\kappa_{1}} \mathrm{HK}_{{\rm p} 0} &\xrightarrow{\kappa_{2}} \mathrm{HK}_{0{\rm p} } \xrightarrow {\kappa_{3}} \mathrm{HK}_{{\rm p} {\rm p} } &
\mathrm{HK}_{0{\rm p} } + \mathrm{RR} &\xrightarrow{\kappa_{4}} \mathrm{HK}_{00} + \mathrm{RR}_{{\rm p} } \\
\mathrm{RR}_{{\rm p} } &\xrightarrow{\kappa_{6}} \mathrm{RR}
 & \mathrm{HK}_{{\rm p} {\rm p} } + \mathrm{RR} &\xrightarrow{\kappa_{5}} \mathrm{HK}_{{\rm p} 0} + \mathrm{RR}_{{\rm p} }\end{align*}
This network describes the phosphorylation of a histidine kinase $\mathrm{HK}$ with two ordered phosphorylation sites and a  response regulator $\mathrm{RR}$. There are in total 6  \emph{species} in the network, $\mathrm{HK}_{00}, \mathrm{HK}_{{\rm p}0},\mathrm{HK}_{0{\rm p}},\mathrm{HK}_{{\rm pp}}, \mathrm{RR}, \mathrm{RR}_{{\rm p }}$. The subscript ${\rm p}$ indicates the phosphorylation status of the sites. Under the assumption of \emph{mass-action kinetics}, the evolution of the species concentrations is modeled by the following system of ordinary differential equations:
        \begin{equation}
        \label{ODE}
                    \begin{aligned}
      &\dot{x}_1 = -\kappa_1 x_1 + \kappa_4 x_3x_5, &\dot{x}_2 = \kappa_1 x_1 - \kappa_2 x_2 + \kappa_5 x_4x_5,   \\ 
  &\dot{x}_3 = \kappa_2 x_2 - \kappa_3 x_3 -\kappa_4 x_3x_5,  &\dot{x}_4 = \kappa_3 x_3 - \kappa_5 x_4x_5,    \\
 &\dot{x}_5 = \kappa_6 x_6 - \kappa_4x_3x_5  - \kappa_5 x_4x_5,   &\dot{x}_6 = -\kappa_6 x_6 + \kappa_4 x_3x_5 +\kappa_5 x_4x_5,  
    \end{aligned}
        \end{equation}
where $x_1 = [\mathrm{HK}_{00}], x_2 =  [\mathrm{HK}_{{\rm p}0}], x_3 = [\mathrm{HK}_{0{\rm p}}], x_4 = [\mathrm{HK}_{{\rm pp}}], x_5 =  [\mathrm{RR}], x_6 =  [\mathrm{RR}_{{\rm p }}]$ denote the concentrations and $\kappa_1, \dots , \kappa_6$ are positive real parameters, called \emph{reaction rate constants}. The set of \emph{positive steady states} of the ODE system~\eqref{ODE} is given by the solutions in $\mathbb{R}^6_{>0}$ of the polynomial equation system obtained by setting the left-hand side of~\eqref{ODE} to zero. This polynomial system can be written as
\[ N \diag(\kappa) x^B = 0,\]
where \begin{align*}
        N = \begin{pNiceArray}{cccccc}
        -1  &  \phantom{-} 0 &   \phantom{-} 0 &   \phantom{-} 1 &   \phantom{-} 0 &  \phantom{-} 0\\
 \phantom{-} 1  &-1&    \phantom{-} 0&    \phantom{-} 0&    \phantom{-} 1&   \phantom{-}  0\\
   \phantom{-} 0 &  \phantom{-}  1 & -1 & -1&   \phantom{-}  0 &  \phantom{-} 0\\
  \phantom{-}  0&   \phantom{-}  0&    \phantom{-} 1 &   \phantom{-} 0&  -1  &  \phantom{-} 0\\
  \phantom{-}  0 &   \phantom{-} 0  & \phantom{-}  0  &-1  &-1 &  \phantom{-}  1\\
   \phantom{-} 0  &  \phantom{-} 0  &  \phantom{-} 0 &   \phantom{-} 1 &   \phantom{-} 1 & -1
     \end{pNiceArray} \in \mathbb{Z}^{6\times 6},   B = \begin{pNiceArray}{cccccc}
     1&  0&  0&  0&  0&  0\\
 0 & 1&  0&  0 & 0 & 0\\
 0 & 0 & 1 & 1 & 0 & 0\\
 0 & 0 & 0 & 0 & 1 & 0\\
 0 & 0 & 0 & 1 & 1 & 0\\
 0 & 0&  0 & 0 & 0 & 1
     \end{pNiceArray} \in \mathbb{Z}^{6\times 6}.
    \end{align*}
The matrix $N$ is called the \emph{stoichiometric matrix}, $B$ is called the \emph{reactant matrix}.

The ODE system~\eqref{ODE} is forward invariant on affine translates of the image of $N$ intersected with $\mathbb{R}^n_{\geq0}$, called \emph{stoichiometric compatibility classes} \cite{Sonntag:ForwardInv,ForInv_Volpert}. To describe these sets, we choose a full rank matrix 
\begin{align*}
      W =  \begin{pNiceArray}{cccccc}
       1& 1& 1& 1& 0& 0\\
     0& 0& 0& 0& 1& 1
     \end{pNiceArray} \in \mathbb{R}^{2\times 6},
    \end{align*}
whose rows form a basis of the left-kernel of $N$. Thus, we can write every stoichiometric compatibility class as
\[  \{ x \in \mathbb{R}^6_{\geq0} \mid Wx = T \},\]
where $T \in \mathbb{R}^2$ is an additional parameter, called \emph{total concentration vector}. Under mild assumptions on the reaction network, for generic choices of the parameters $\kappa,T$, the equation system
\begin{align}
\label{Eq:SteadyStatecConv}
N\diag(\kappa) x^B = 0, \qquad Wx = T
\end{align}
has a finite number of solutions in $(\mathbb{C}^*)^6$  \cite{FeliuDim}. If, for some choice of the parameters, \eqref{Eq:SteadyStatecConv} has at least two positive real solutions, the network is said to be \emph{multistationary}. It is known that for our example, the hybrid histidine kinase network, there are choices of the parameters where \eqref{Eq:SteadyStatecConv} has three positive solutions~\cite{HHK,PLOS_IdParaRegions}. In general, determining whether a reaction network is multistationary and identifying parameters for which \eqref{Eq:SteadyStatecConv} has many positive solutions is an active area of research in reaction network theory \cite{shiu:survey, PLOS_IdParaRegions, bihan:regions2,bihan:regions3,bihan:regions,MultStrucRN,conradi:cones,MultiSite_Phosph_Dicken,markev}.

In the following, we discuss how this problem can be addressed using tropical geometry. We consider the matrices
\begin{align*}
 {\footnotesize   C := \begin{pNiceArray}{c|c|c}
     N & {\bf 0}_{6 \times 6} & {\bf 0}_{ 6\times 1} \\
     \hdottedline
    {\bf 0}_{2\times 6} & W & -T
 \end{pNiceArray} \in \mathbb{R}^{8\times 13},  \quad A := \begin{pNiceArray}{c|c|c}
     B & \Id_{6\times 6} & {\bf 0}_{6\times 1} \\
 \end{pNiceArray} \in \mathbb{R}^{6\times 13}, \quad K = \begin{pmatrix}
     \kappa\\
     \mathds{1}
 \end{pmatrix} \in \mathbb{R}^{13}}.
\end{align*}
Here, we denote by ${\bf 0}_{k \times \ell}$ matrices whose entries are zero and by  $\mathds{1}$ the vector of size $7$ whose entries are equal to one. Using this notation, we rewrite~\eqref{Eq:SteadyStatecConv} as 
\begin{align}
\label{Eq:RewrittenSystem}
 C \diag(K) x^A = 0
\end{align}
We fix $T = (10,20)$ and set $\kappa = t^h$, for $h = (7, -6, -2, -3, -3, 3)$ and $t > 0$. 
The equation system \eqref{Eq:RewrittenSystem} is tropically transverse and we have that
\begin{align}
\label{Eq:IntersGivingLowerBound}
 \big(\Trop^+\big( \ker (C) \big) - (h,{\bf 0}_{1 \times 7}) \big) \cap \rowspan(A)
\end{align}
consists of three points. Thus, by  \Cref{Prop::LowerBoundVertic}, for sufficiently small values of $t > 0$, the hybrid histidine kinase network has at least three positive steady states for the parameters $\kappa = t^h$ and $T=(10,20)$. For the computation, we used the following \texttt{Oscar.jl} \cite{OSCARbook} code:

\begin{verbatim}
using Oscar

R, x = polynomial_ring(QQ, "x"=>1:13)
TropL_positive = positive_tropical_variety(ideal(C*x),tropical_semiring_map(QQ))
rowspan_A = convex_hull(Vector{Int}(zeros(13)), nothing, [ A[i,:] for i in 1:6])

h = Vector{Int}([7, -6, -2, -3, -3, 3, 0, 0, 0, 0, 0, 0, 0])
pos_intersection_points = []
for max_cone in maximal_polyhedra(TropL_positive)
    intersection_point = intersect(max_cone + convex_hull(-h),rowspan_A)
    if dim(intersection_point) == 0
        push!(pos_intersection_points,intersection_point)
    end
end
\end{verbatim}
The above code simply iterates through all maximal cones of $\Trop^+\big( \ker (C) \big) - (h,{\bf 0}_{1 \times 7})$ and computes their intersections with  $\rowspan(A)$. If the intersection is zero-dimensional, then the two tropical varieties intersect transversally at that point.

\begin{remark}
    We are not aware of any effective method for finding the $h$ that yields the highest possible lower bound. In our experiments, we picked $h$ randomly and computed the number of intersection points in~\eqref{Eq:IntersGivingLowerBound}. Developing an effective algorithm to maximize the number of intersection points is an interesting problem, but it lies beyond the scope of this paper.
\end{remark}

\color{black}

\paragraph{\textbf{Acknowledgements. }}
The authors thank Yue Ren for useful discussions and for pointing us to \cite[Lemma 2.7]{Bossinger2017}. MLT thanks Elisenda Feliu for the inspiring discussions and comments on the manuscript. MLT was supported by the European Union under the Grant Agreement no. 101044561, POSALG. Views and opinions expressed are those of the author(s) only and do not necessarily reflect those of the European Union or European Research Council (ERC). Neither the European Union nor ERC can be held responsible for them.

{\small

}

\appendix
\section{o-minimal geometry}
\label{AppendixA}
In this appendix, we provide more detail on the construction of the Hardy field. Our exposition closely follows that in \cite{Alessandrini::LogLimit,Ebbinghaus::Logic}.
 
A set of \emph{symbols} $S$ consists of three types of symbols: \emph{$n$-ary relation symbols}, \emph{$n$-ary function symbols}, and \emph{constants} \cite[Chapter II, Definition 2.1]{Ebbinghaus::Logic}. We denote by $L_S$ the corresponding \emph{first order language} \cite[Chapter II, Definition 3.2]{Ebbinghaus::Logic}. An $S$-structure $(R,\mathfrak{a})$ is a pair such that $R$ is a non-empty set and $\mathfrak{a}$ is a map on $S$ such that for every $n$-ary relation symbol $r \in S$, $\mathfrak{a}(r)$ is an $n$-ary relation on $R$, for every $n$-ary function symbol $f \in S$, $\mathfrak{a}(f)\colon R^n \to R$ is function, and for every constants $c \in S$, $\mathfrak{a}(c) \in R$. An \emph{expansion} of an $S$-structure $(R,\mathfrak{a})$ is an  $S'$-structure $(R,\mathfrak{a}')$  such that $S \subset S'$ and $\mathfrak{a}'(s) = \mathfrak{a}(s)$ for all $s \in S$ \cite[Chapter III, Definition 4.7]{Ebbinghaus::Logic}.

For our purposes, the most important example of an $S$-structure is $\overline{\mathbb{R}} = (\mathbb{R},\mathfrak{a})$, where $S = \{ <,+,-,\cdot,0,1\}$ and $\mathfrak{a}$ is the usual interpretation of the symbols in $S$. We are also interested in the expansion $\mathbb{R}_{\text{an}^*}$ of  $\overline{\mathbb{R}}$ where the set of symbols $S'$ contains additionally function symbols given by locally convergent generalized real Puiseaux series~\cite{Dries::GenPowerSeries}.

Given an $S$-structure $(R,\mathfrak{a})$, a subset of $R^n$ is \emph{definable} if it is given by an $L_S$-formula $\varphi(x_1,\dots,x_n,y_1,\dots,y_m)$ and by parameters $a_1,\dots,a_m \in R$. A function is \emph{definable} if its graph is definable. An  $S$-structure $(R,\mathfrak{a})$ is \emph{o-minimal} if the definable subsets of $R$ are the finite unions of points and intervals \cite[Definition 1.4]{Coste::oMinimal}. For example, for $S = \{ <,+,-,\cdot,0,1\}$ the $S$-structure $\overline{\mathbb{R}}$ is o-minimal, the definable sets (resp. functions) are exactly the semi-algebraic sets (resp. functions).

In  \cite[Theorem A]{Dries::GenPowerSeries}, it has been showed that $\mathbb{R}_{\text{an}^*}$ is o-minimal, moreover the definable functions have the following form. 

\begin{theorem}
    \cite[Theorem B]{Dries::GenPowerSeries} Let $\epsilon >0$. A function $f\colon (0,\epsilon) \to \mathbb{R}$ is definable in $\mathbb{R}_{\text{an}^*}$ if and only if there exists $\delta \in (0,\epsilon)$ and a locally convergent generalized real Puiseaux series $x(t)$ as in \eqref{Eq::GenRealPuisSeries} such that $f(t) = x(t)$ for all $t \in (0,\delta)$.
\end{theorem}

Thus, we can think of definable functions in $\mathbb{R}_{\text{an}^*}$ as locally convergent generalized real Puiseaux. Now, we recall the definition of the Hardy field \cite[Section 4.1]{Alessandrini::LogLimit}, \cite[Section 5.3]{Coste::oMinimal}.

\begin{definition}
\label{Def:HardyField}
    Two functions $f,g\colon (0,\epsilon) \to \mathbb{R}$ which are definable in $\mathbb{R}_{\text{an}^*}$ are equivalent if and only if there exists $\delta \in (0,\epsilon)$ such that $f(t) = g(t)$ for all $t \in (0,\delta)$. The Hardy field $H( \mathbb{R}_{\text{an}^*})$ is defined as the set of equivalence classes under this equivalence relation.
\end{definition}

\end{document}